\documentclass{amsart}

\usepackage{amsthm}
\usepackage[leqno]{amsmath}
\usepackage{latexsym,amsfonts,amssymb,mathtools}
\usepackage{calligra}
\usepackage[T1]{fontenc}
\usepackage{cite}
\usepackage{array}
\usepackage{url}
\usepackage{relsize}
\usepackage{enumerate}
\numberwithin{equation}{section}

\usepackage{stackrel}

\usepackage[all,cmtip]{xy} \SelectTips{eu}{} \SilentMatrices
\xyoption{all}
\xyoption{arc}

\usepackage[urlcolor=blue]{hyperref}

\usepackage{color}

\usepackage{hyperref}

\newcommand{\numberseries}{\mdseries}   

\newlength{\thmtopspace}                
\newlength{\thmbotspace}                
\newlength{\thmheadspace}               
\newlength{\thmindent}                  

\renewcommand{\subparagraph}{\vspace{\thmbotspace}}

\newtheoremstyle{bfupright head,slanted body}
{\thmtopspace}{\thmbotspace}
{\slshape}{\thmindent}{\bfseries}{.}{\thmheadspace}
{{\numberseries \thmnumber{\bf #2 }}\thmnote{#3}}

\newtheoremstyle{bfupright head,upright body}
{\thmtopspace}{\thmbotspace}
{\upshape}{\thmindent}{\bfseries}{.}{\thmheadspace}
{{\numberseries \thmnumber{\bf #2 }}\thmnote{#3}}

\newtheoremstyle{bfit head,upright body}
{\thmtopspace}{\thmbotspace}
{\upshape}{\thmindent}{\upshape}{.}{\thmheadspace}
{{\numberseries\thmnumber{\bf #2 }}
  {\bfseries\itshape\thmnote{\negthickspace#3}}}

\newtheoremstyle{it head,upright body}
{\thmtopspace}{\thmbotspace}
{\upshape}{\thmindent}{\upshape}{.}{\thmheadspace}
{{\numberseries\thmnumber{\bf #2 }}
  {\itshape\thmnote{\negthickspace#3}}}

\newtheoremstyle{fixed bf head,slanted body}
{\thmtopspace}{\thmbotspace}{\slshape}
{\thmindent}{\bfseries}{.}{\thmheadspace}
{{\numberseries \thmnumber{\bf  #2 }}\thmname{#1}\thmnote{ (#3)}}

\newtheoremstyle{fixed bf head,upright body}
{\thmtopspace}{\thmbotspace}{\upshape}
{\thmindent}{\bfseries}{.}{\thmheadspace}
{{\numberseries \thmnumber{\bf #2 }}\thmname{#1}\thmnote{ (#3)}}

\newtheoremstyle{fixed bfit head,upright body}
{\thmtopspace}{\thmbotspace}{\upshape}
{\thmindent}{\bfseries\itshape}{.}{\thmheadspace}
{{\numberseries \thmnumber{\bf#2 }}\thmname{#1}\thmnote{ (#3)}}

\newtheoremstyle{sc head,small body}
{\thmtopspace}{\thmbotspace}
{\small\upshape}{\thmindent}{\scshape}{.}{\thmheadspace}
{\thmname{#1}}

\newtheoremstyle{numbered paragraph}
{\thmtopspace}{\thmbotspace}{\upshape}
{\thmindent}{\upshape}{}{0pt}
{{\numberseries \thmnumber{\bf #2 }}}

\newtheoremstyle{unnumbered paragraph}
{\thmtopspace}{\thmbotspace}{\upshape}
{\parindent}{\upshape}{}{0pt}

\theoremstyle{bfupright head,slanted body}
\newtheorem{res}{}[section]             \newtheorem*{res*}{}

\theoremstyle{bfit head,upright body}
                 \newtheorem*{com*}{}

\theoremstyle{bfupright head,upright body}
\newtheorem{bfhpg}[res]{}               \newtheorem*{bfhpg*}{}

\theoremstyle{it head,upright body}
               \newtheorem*{ithpg*}{}

\theoremstyle{sc head,small body}

\theoremstyle{fixed bf head,slanted body}
\newtheorem{thm}[res]{Theorem}          \newtheorem*{thm*}{Theorem}
\newtheorem{prp}[res]{Proposition}      \newtheorem*{prp*}{Proposition}
        \newtheorem*{cor*}{Corollary}
\newtheorem{lem}[res]{Lemma}            \newtheorem*{lem*}{Lemma}

\theoremstyle{fixed bf head,upright body}
\newtheorem{dfn}[res]{Definition}       \newtheorem*{dfn*}{Definition}
     \newtheorem*{con*}{Construction}
      \newtheorem*{obs*}{Observation}
\newtheorem{rmk}[res]{Remark}           \newtheorem*{rmk*}{Remark}
          \newtheorem*{exa*}{Example}
         \newtheorem*{exe*}{Exercise}
            \newtheorem{stp*}{Setup}
            \newtheorem{blk*}{\!\!\!}

\newtheorem{ntn}[res]{Notation}

\theoremstyle{numbered paragraph}
\newtheorem{ipg}[res]{}

\theoremstyle{unnumbered paragraph}
\newtheorem{ipg*}{}

\newlength{\thmlistleft}        
\newlength{\thmlistright}       
\newlength{\thmlistpartopsep}   
\newlength{\thmlisttopsep}      
\newlength{\thmlistparsep}      
\newlength{\thmlistitemsep}     

\newcounter{eqc}
  {\end{list}}%

\newcounter{prt}
  {\end{list}}%

\newcounter{rqm}
  {\end{list}}%

  {\end{list}}%

  {\end{list}}%

\newcounter{exercise}
  {\end{list}}%

\newenvironment{prf*}[1][Proof]{%
  \begin{proof}[\it #1]
    \setcounter{equation}{0}
    \renewcommand{\theequation}{\arabic{equation}}}
  {\end{proof}
}

\newcommand{\pgref}[1]{(\ref{#1})}

\renewcommand{\eqref}[1]{\pgref{eq:#1}}

\numberwithin{equation}{res}
\newcounter{marcom}

\newcommand{\coker}{\nobreak{\operatorname{coker\,}}}

\renewcommand{\ge}{\geqslant}
\newcommand{\onto}{\twoheadrightarrow}
\newcommand{\lra}{\longrightarrow}
\newcommand{\xra}[2][]{\xrightarrow[#1]{\;#2\;}}

\newcommand{\fm}{\mathfrak{m}}

\newcommand{\Hom}{\operatorname{Hom}}

\newcommand{\longleftrightarrows}{\leftrarrows\joinrel\Rrelbar\joinrel\lrightarrows}
\newcommand{\leftrarrows}{\mathrel{\raise.75ex\hbox{\oalign{%
   $\scriptstyle\leftarrow$\cr
   \vrule width0pt height.5ex$\hfil\scriptstyle\relbar$\cr}}}}
 \newcommand{\lrightarrows}{\mathrel{\raise.75ex\hbox{\oalign{%
  $\scriptstyle\relbar$\hfil\cr
   $\scriptstyle\vrule width0pt height.5ex\smash\rightarrow$\cr}}}}
 \newcommand{\Rrelbar}{\mathrel{\raise.75ex\hbox{\oalign{%
   $\scriptstyle\relbar$\cr
  \vrule width0pt height.5ex$\scriptstyle\relbar$}}}}
\newcommand{\im}{\operatorname{im}}

\renewcommand{\Gamma}{\textrm{C}}

\newcommand{\kos}{\textrm{kos}}
\newcommand{\tr}{\textrm{tr}}

\begin{document}
  
  \large   
  
  \allowdisplaybreaks[2]
  
  \title{Transferring algebra structures on complexes}
  
  \author[C.\, Miller  ]{Claudia Miller}
  \address{C.M.\newline\hspace*{1em} Mathematics Department, Syracuse University, Syracuse, NY 13244, U.S.A.}
  \email{clamille@syr.edu} 
  \urladdr{http://clamille.mysite.syr.edu}

  \author[H. Rahmati]{Hamidreza Rahmati}  
  \address{H.R.\newline\hspace*{1em} Department of Mathematics,  University of Nebraska,
    Lincoln, NE 68588,  USA}  
  \email{hrahmati2@unl.edu}

  
  

\thanks{C.\ Miller partially supported by the National Science Foundation (DMS-1003384),  Syracuse University Small Grant, and Douglas R.\ Anderson Faculty Scholar Fund.}

  \begin{abstract}
  We discuss a homological method for transferring algebra structures on complexes along suitably nice homotopy equivalences, including those obtained after an application of the Perturbation Lemma. 
  We study the implications for the Homotopy Transfer Theorems under such homotopy equivalences. 
  
  As an application, we discuss how to use the homotopy on a Koszul complex given by a scaled de Rham map to find a new method for building a dg algebra structure on a well-known resolution, obtaining one that is both concrete and permutation invariant.
 
  \end{abstract}

  \maketitle
 
 \section*{Introduction}

In this paper we study descent of algebra structures on complexes along a suitably nice maps, discussing a method that allows one to transfer an algebra structure on a complex to another complex. More precisely, we show that a differential graded (dg) algebra structure on one complex can be transferred to another complex that is a deformation retract of it with a homotopy that satisfies a generalized version of the Leibniz rule and another mild hypothesis that is frequently imposed; see Proposition~\ref{descent-dga}.

For our main application, 
a homological tool called the Perturbation Lemma is the key. For this purpose, we also prove that the analogous dg algebra descent result holds even after applying the Perturbation Lemma as long as the complex remains a dg algebra after perturbation; see Proposition~\ref{descent-dga-perturbation}. 

Such descent results turn out to be instances of the well-known Homotopy Transfer Theorem (HTT) which usually yields only an $A_\infty$-algebra structure on the retract. Our extra hypothesis on the homotopy ensures that the product on the retract is, in fact, associative. We also make a detour to show that with these hypotheses the descended higher $A_\infty$-operations vanish as well and that if one starts with an $A_\infty$-algebra instead of a dg algebra then under an analogous stronger condition on the homotopy, the descended higher operations are much simpler than usual; see Propositions~\ref{descent-dga-Ainfinity} and ~\ref{descent-Ainfinity}.

We are mainly interested in algebra structures on {\em minimal} free resolutions of algebras. 
It is worth noting that few such resolutions are known to carry dg algebra structures. 
And yet having one provides one with a powerful tool. 
Short resolutions are known to have a dg algebra structure \cite{Her-74}, \cite{BuchEis-77}, \cite{KusMil-80}, \cite{Kus-87} but counterexamples of longer ones can be found in \cite{Avr-81}, \cite{Sri-92}, \cite{Sri-96}. 
We refer the reader to \cite{Av-96} for a full discussion. 

In the second portion of the paper, we apply our descent results to obtain a new method of building dg algebra structures that are both concrete and highly symmetric  on some well-known resolutions.
These resolutions, constructed by Buchsbaum and Eisenbud in \cite{BuchEis-75} using Schur modules, are the minimal graded free resolutions $\mathbb L_a$ of the quotients $R/\fm^a$ of a polynomial ring $R$ in $n$ variables by the $a$th powers of the homogeneous maximal ideal. 

These resolutions were shown to have a dg algebra structure by several authors. The first was Srinivasan in 1989 who put an explicit product using Young tableaux; see \cite{Srin89}. Next in 1996 Peeva proved in \cite{Peev-96} that one can place a dg algebra structure on the Eliahou-Kervaire resolution, which applies here since the powers of the maximal ideal are Borel-fixed. 
In \cite{Mae-01} Maeda used the representation theory of the symmetric group $S_n$ to show that in characteristic zero any $S_n$-invariant lift 
of the multiplication on the quotient ring to the resolution is automatically associative, but did not give any explicit formulas. 
Our goal is to find a concrete product that is also  $S_n$-invariant; however, we pay the price of that by having more terms and requiring  coefficients in the rational numbers. 

In contrast to the others, we use our homotopy transfer results to define a product 
which is both explicit and very natural in that it is descended from a truncation of a Koszul complex and naturally $S_n$-invariant. 
One benefit of defining a product using our method is that it  explains why there should really be a dg algebra structure here and is hopefully more canonical and hence useful for further applications. 
Another is that it enables us to define dg algebra homomorphisms between these resolutions. 
Our product works both in characteristic zero and in positive characteristics larger than $n+a$. 
For this, we use a scaled version of the deRham differential to produce a contracting homotopy for the Koszul complex on the variables that satisfies the generalized Leibniz rule; see Lemma~\ref{scaledLeibniz}. It is also worth mentioning that the product that we define can be constructed in a basis free fashion; see Remark~\ref{basisfree}.

More precisely, we prove the following results.
In the first one, using the Perturbation Lemma to obtain the resolution of $\mathbb L_a$ of $R/\fm^a$ as a retract of a complex with a dg algebra structure, we transfer the algebra structure as follows. 
For this, let $\mathbb S=\{ \Lambda^i \otimes S_j \mid i,j \geq 0 \}$ be the totalization of the double complex defined in (\ref{chunk-tautological}), and let $\mathbb X_a$ be the quotient of the dg algebra $\mathbb S$ by the dg ideal $\tr_{\geq a}(\mathbb S)=
    \{ \Lambda^i \otimes S_j \mid j \geq a \}$.

\vspace{2mm}

\noindent 
{\bf Theorem~\ref{thm-L}.}
\emph{
Let $a$ be a positive integer. Suppose that $k$ is a field of characteristic zero or positive characteristic $p\geq a+n$. 
Consider the deformation retract
\[
\mathbb X_a 
\stackrel[p_\infty]{i_\infty}{\longleftrightarrows}
\mathbb L_a
\]
obtained in (\ref{chunk-PL}). 
Defining the product of  $\alpha,\beta\in \mathbb L_a$ by \begin{equation*}
     \alpha \beta 
= p_\infty 
\left( 
i_\infty(\alpha)  i_\infty(\beta)
\right) 
  \end{equation*} 
yields a dg algebra structure on $\mathbb L_a$. Furthermore, with this structure the map $i_\infty$ is a homomorphism of dg algebras.
}
 
\vspace{2mm} 

The second result gives some very natural dg algebra homomorphisms between these resolutions. 

\vspace{2mm}

\noindent 
{\bf Theorem~\ref{dghomomorphism}.}
\emph{
Let $a$ and $b$ be positive integers with $b \geq a$ and let \[
\pi_{b,a}\colon
\mathbb X_b
\onto 
\mathbb X_a
\] 
be the natural surjection. The chain map  
 \[
 f_{b,a}
 =p_{\infty}\pi_{b,a}i_{\infty}
 \colon \mathbb L_b \to \mathbb L_a
 \] 
 is a homomorphism of dg algebras that gives a lifting of the natural surjection $R/\fm^b \lra R/\fm^a$. In particular, the Koszul complex on the variables, which is $\mathbb L_1$,  is a dg algebra over $\mathbb L_b$ for every positive integer $b$. 
 }
 
\emph{
Moreover, if $c\geq b \geq a$ then $f_{c,a}=f_{b,a} f_{c,b}$.}

\vspace{2mm}

%

The paper is organized as follows: In Section~\ref{sec-transfer}, we prove the general statement for descending dg algebra structures along a special deformation retract whose homotopy satisfies the generalized Leibniz rule; see Proposition~\ref{descent-dga}. 
In Section~\ref{sec-Ainfinity}, we make a detour to study the implications of the additional hypothesis of generalized Leibniz-type rules for the Homotopy Transfer Theorems, resulting in Propositions~\ref{descent-dga-Ainfinity} and ~\ref{descent-Ainfinity}. 
In Section~\ref{sec-perturbation}, we recall the well-known Perturbation Lemma and prove a more general version of our dg algebra descent result from Section~\ref{sec-transfer}, namely Proposition~\ref{descent-dga-perturbation}. 
Section~\ref{sec-L} contains our main application, which is to obtain dg algebra structures on the resolutions $\mathbb L_a$ of Buchsbaum and Eisenbud for all $a\geq 1$. In Section~\ref{sec-homomorphism} we obtain dg algebra homomorphisms between these resolutions. 

In this paper, we assume that the  complexes consist of $R$-modules for some commutative ring $R$ and that they are graded homologically, rather than cohomologically; the reader should be aware that most sources for $A_{\infty}$ algebras use the latter instead. 
One more note: 
All double complexes in this paper are considered to be anticommutative as in \cite{Wei-94}, and hence their totalizations do not require any change of sign in the differentials. 
When we speak of a double complex  (for example, when we discuss a dg algebra structure on it) we mean the totalization of it as a complex; which way we are viewing it should be clear from the context. 



 \section{Transfer of dg algebra structures}
 \label{sec-transfer}

In this section, we show how to transfer dg algebra structures along certain homotopy equivalences, namely deformation retracts whose associated homotopy behaves well with respect to products.
We compare this in the next section to the Homotopy Transfer Theorem, via which the dg algebra structure descends to an $A_\infty$-algebra structure.  We also discuss there what happens when the original complex is merely an $A_\infty$-algebra. 

 \begin{dfn}
 \label{def-dga}
  A \emph{differential graded algebra over} $R$ (dg algebra) is a complex $(X, \partial)$ of
  $R$-modules lying in nonnegative degrees equipped with a product given by a chain map     \begin{equation*}
    X \otimes_R X \to X, ~ (\alpha, \beta) \to \alpha\beta
  \end{equation*}
giving an associative and unitary product with $1\in X_0$. 

The fact that the product is a chain map is equivalent to the differentials of $X$ satisfying the \emph{Leibniz rule}:  \begin{equation*}
    \partial(\alpha\beta) = \partial(\alpha)\beta+(-1)^{|\alpha|}\alpha\partial(\beta), ~\text{for
      all}~ \alpha,\beta \in X 
\end{equation*}
where $|\alpha|$ denotes the degree of $\alpha$. In addition, we assume that the product is \emph{strictly graded commutative}, that is,  $\alpha\beta  =(-1)^{|\alpha||\beta|}\beta\alpha$ for all $\alpha, \beta \in X$ and $\alpha^2=0$ if the degree of $\alpha$ is odd.

A \emph{homomorphism} of dg algebras is a morphism of complexes  $\phi \colon X \to X'$ such that $\phi(1)=1$ and $\phi(\alpha\beta)=\phi(\alpha)\phi(\beta)$.
\end{dfn}

We now recall the definitions of the main ingredients for transferring algebra structures.

\begin{dfn}
\label{DR}
A set of \emph{homotopy equivalence data} between two chain complexes is the following set of information: quasi-isomorphisms of complexes 
\[
(X,\partial^X) 
\stackrel[p]{i}{\longleftrightarrows}
(Y,\partial^Y)
\]
with $ip \simeq 1$ via a chosen homotopy $h$ on $X$, that is $ip=1+\partial^Xh+h\partial^X$. 
Note that sometimes a homotopy equivalence is defined to include the condition $pi\simeq 1$, but the version of the Perturbation Lemma in the sources we consulted do not include this condition. 

It is called a {\it deformation retract} if, in addition, one has $pi=1$. 
This condition holds in our applications.

\end{dfn}

Next we see that deformation retracts that satisfy $hi=0$ give a simple way to transfer dg algebra structures from a complex to its summand as long as the associated homotopy satisfies the following additional property, a weakening of the Leibniz rule, which we introduce below.

\begin{dfn}
 \label{def-scaled-Leibniz}
Let $X$ be a complex of $R$-modules equipped with a product, and let $h\colon X \lra X$ be a graded map (but not necessarily a chain map). 
We say that $h$ satisfies the \emph{generalized Leibniz rule} if one has 
\begin{equation*}
    h(\alpha\beta) \subseteq h(\alpha)X+ X h(\beta)
  \end{equation*}
for every $\alpha$ and $\beta$ in $X$. 

For our application, the map $h$ will in fact satisfy a stronger condition, which we call \emph{scaled Leibniz rule}, namely that for every $\alpha ,\beta \in X $ there are  $r,s \in R$ depending only on the degrees of $\alpha$ and $\beta$, respectively, such that \begin{equation*}
    h(\alpha\beta) = r h(\alpha)\beta+ s\alpha h(\beta)
  \end{equation*}
\end{dfn}

We now prove the main result of this section.

\begin{prp}
\label{descent-dga}
Let $X$ be a dg algebra. 
Consider a deformation retract
\[
(X,\partial^X) 
\stackrel[p]{i}{\longleftrightarrows}
(Y,\partial^Y)
\]
with associated homotopy $h$ that satisfies the generalized Leibniz rule and $hi=0$. The following product defines a dg algebra structure on $Y$ 
\[
\alpha  \beta 
\stackrel{\text{ def }}{=} 
p \left( i(\alpha)  i(\beta) \right)
{\text{ for }}~
\alpha, \beta \in Y
\]
where the product inside parentheses is the one in $X$. 

Moreover, with this structure on $Y$, the map $i$ becomes a dg algebra homomorphism.
\end{prp}

\begin{proof} 
The Leibniz rule for $Y$ holds without the assumptions that $h$ satisfies the generalized Leibniz rule and $hi=0$. 
Indeed, for any  elements $\alpha, \beta \in Y$, one has 
\begin{align*}
\partial^Y(\alpha  \beta)
&= (\partial^Y  p) (i(\alpha)  i(\beta) ) \\
&= (p  \partial^X) (i(\alpha)  i(\beta) ) \\
&= p \left(\partial^X (i(\alpha))  i(\beta) 
+ (-1)^{|\alpha|} 
i(\alpha)  \partial^X (i(\beta)) \right) \\
&= p \left(i(\partial^Y (\alpha))  i(\beta) \right)
+ (-1)^{|\alpha|} 
p \left( i(\alpha)  i(\partial^Y (\beta)) \right) \\
&=\partial^Y (\alpha)  \beta 
+ (-1)^{|\alpha|} 
\alpha  \partial^Y (\beta) 
\end{align*}
where the first equality is from the definition of the product, the second one holds since $p$ is a chain map, the third one is from the Leibniz rule for $X$, the fourth holds since $i$ is a chain map, the fifth is again from the definition of the product. 

To prove the associativity and the last assertion, we first show that 
\begin{equation}
\label{vanishing}
   (\partial^Xh+h\partial^X) \left ( i (\alpha)  i(\beta)\right )=0
\end{equation}
for all $\alpha, \beta \in Y$. If one expands this expression  using that  $\partial^X$ satisfies the Leibniz rule and $h$ satisfies the generalized Leibniz rule, one sees that every term has a factor with $hi$, which is zero, or a factor with $h\partial^Xi$ which is also zero since $\partial^Xi=i\partial^Y$. 

To verify associativity, take any elements $\alpha, \beta, \gamma \in Y$. One has 
\begin{align*}
(\alpha  \beta)  \gamma 
&=  p \left( i (\alpha)  i(\beta)\right )  \gamma \\
&= p \left( (i  p ( i (\alpha)  i(\beta) ) i(\gamma)  \right) \\
&= p \left( (1+\partial^Xh+h\partial^X) \left ( i (\alpha)  i(\beta)\right )  i(\gamma)  \right ) \\
&= p \left(  \left( i (\alpha)  i(\beta)\right)  i(\gamma)  \right) 
\end{align*}
where the first two equalities are from the definition of the product and  the third one is by the equality $ip=1+\partial^Xh+h\partial^X$.  
A similar argument shows that 
\begin{align*}
 \alpha  (\beta  \gamma) 
 = 
 p \left(   i (\alpha) \left( i(\beta)  i(\gamma) \right) \right)
\end{align*}
and hence associativity holds since it holds for $X$. 

Finally, one can see that $p(1)$ is the identity element of $Y$ and that the product on $Y$ is graded commutative since $p$ and $i$ are graded maps.

To see that the map $i$ is a dg algebra homomorphism, let
 $\alpha, \beta \in Y$. One then has 
\begin{align*}
i(\alpha\beta) 
&=
i p 
\left( 
i(\alpha)i(\beta) 
\right)
\\
&=
(1+\partial^Xh+h\partial^X)
\left( 
i(\alpha)i(\beta) 
\right)\\
&= i(\alpha)i(\beta) 
\end{align*}
where the second equality holds as $i$ and $p$ form a deformation retract and the last one follows from (\ref{vanishing}).
\end{proof}

Note that the condition $hi=0$ in Proposition~\ref{descent-dga} holds when the deformation retract is special; see \ref{specialDR} for definition.

For the application we have in mind in Section~\ref{sec-L}, we need a slightly stronger result since, after we apply the Perturbation Lemma, the new homotopy need no longer satisfy the generalized Leibniz rule even if the original one does; however we show in this case that the descent still works as long as the original deformation retract is special. 
We give this result in Proposition~\ref{descent-dga-perturbation}.

\section{Connections to Homotopy Transfer Theorems}
\label{sec-Ainfinity}

In Proposition~\ref{descent-dga}, we found that a dg algebra structure descends along certain deformation retracts as long as the homotopy satisfies the generalized Leibniz rule, defined in \ref{def-scaled-Leibniz}. In this section, we compare this to the result of the well-known Homotopy Transfer Theorem, via which the dg algebra structure descends to an $A_\infty$-structure, which can have nontrivial higher products even when the descended structure is associative and hence a dg algebra. 
Under the aforementioned additional hypothesis on the homotopy, we compute the higher operations that arise from the descent and find them to vanish after all in Proposition~\ref{descent-dga-Ainfinity}. In Proposition~\ref{descent-Ainfinity}, we also discuss what happens when the original complex is merely an $A_\infty$-algebra under a similar, but much stronger hypothesis on the homotopy. 

We note that most sources for the Homotopy Transfer Theorems work with dg algebras and $A_\infty$-algebras over a field of characteristic zero. 
However, these are known to hold over a commutative ring $R$ as long as one  makes some freeness assumptions. For simplicity we assume in this section that $R$ is a field of characteristic 0 (or, more generally, that we are in a setting in which the Homotopy Transfer Theorems are known to hold). 
However, we should point out that our transfer results Proposition~\ref{descent-dga} and ~\ref{descent-dga-perturbation} do not require any such hypotheses. 

We begin by recalling both the definition of an $A_\infty$-algebra and the Homotopy Transfer Theorem for a dg algebra. 
The concept of an $A_\infty$-algebra was introduced by Stasheff in \cite{Sta-63} in his study of loop spaces, where the natural product is only associative up to homotopy. 
For some expositions of this topic, see \cite{Kel-2001}, \cite{Kel-06}, and \cite{Lef-03}.

\begin{dfn}\label{def-Ainfinity}
An \emph{$A_\infty$-algebra} over a ring $R$ is a complex $A$ of $R$-modules together with $R$-multilinear maps of degree $n-2$ 
\[
m_n \colon A^{\otimes n} \to A
\] 
for each $n\geq 1$, called operations or multiplications, satisfying the following relations, called the Stasheff identities. 
\begin{itemize}
\item 
The first operation is simply the differential:  
\[
m_1=\partial_A.
\]
\item 
The second operation satisfies the Leibniz rule:  
\[
m_1 m_2 = m_2 (m_1 \otimes 1 + 1 \otimes m_1).
\]
\item 
The third one verifies that $m_2$ is associative up to  the homotopy $m_3$: \begin{align*}
m_2(&1 \otimes m_2 - m_2 \otimes 1)  \\ &=m_1m_3 +m_3(m_1 \otimes 1 \otimes 1+1 \otimes m_1 \otimes 1+1 \otimes 1 \otimes m_1)
\end{align*}
Note that the left hand side is the obstruction to associativity for $m_2$ and that the right hand side is the boundary of $m_3$ in $\Hom_R(A^{\otimes 3},A)$. 
\\[-1mm]
\item 
More generally, for $n \geq 1$, we have
\[
\sum_{s=1}^n
\sum_{r,t\ge 0}  (-1)^{r+st} 
m_{r+1+t} (1^{\otimes r} \otimes m_s \otimes 1^{\otimes t}) = 0
\]
where the sums are taken over the values of $r,s,t$ with $r + s + t=n$. 
\end{itemize}

Note that when one applies the maps in each formula above 
to an element, one should use the Koszul sign rule: 
For graded maps $f$ and $g$, one has 
\[
(f \otimes g)(x \otimes y) = (-1)^{|g||x|}f(x) \otimes g(y)
\]
for homogeneous elements $x$ and $y$, where $|w|$ denotes the degree of $w$ whether it is a map or homogenous element. 
\end{dfn}

Recall that we are using homological notation; in cohomological notation the degree of $m_n$ would be $2-n$ rather than $n-2$. 
Note also that for the signs we follow the conventions in Getzler-Jones \cite{GetJon-90}; see, for example, the survey by Keller \cite{Kel-2001}.

\begin{rmk}\label{A-dga}
Note that an $A_\infty$-algebra  whose operations $m_n$ are zero for all $n\ge 3$ is a dg algebra where the product is given by $m_2$. 
Conversely, a dg algebra can be given the structure of an $A_\infty$-algebra by setting $m_{\geq 3}=0$. 

However, $m_1$ and $m_2$ can usually be extended to other $A_\infty$-algebra structures. Indeed, one can have nonzero higher operations for which the boundary of $m_3$ in $\Hom_R(A^{\otimes 3},A)$ is equal to zero and hence $A$ is still associative. \end{rmk}

The Homotopy Transfer Theorems were first proved by Kadeishvili in \cite{Kad-80} and \cite{Kad-82}. We recall them in (\ref{HTT}) and (\ref{HTT-Ainfinity}). For this, we follow the exposition in Vallette's survey \cite{Vall-14}. 
We note that our signs are the opposite of those in his survey since his homotopy is the negative of ours (he has $1-ip=\partial^Xh+h\partial^X$, rather than $ip-1$). This should not make a difference as the precise signs do not matter for our proofs.

\begin{ntn} 
We introduce the planar rooted tree notation from Vallette's survey to represent these products pictorially as this will make it easier to describe the Homotopy Transfer Theorem. 
All the diagrams are read from the top down, that is, the inputs are thought of as being entered on the top and the multi-intersections correspond to the higher products $m_n$ being performed. 
Further, wherever a letter appears in such a diagram, one applies the corresponding map at that point. 
Again, the sign rule described in Definition~\ref{def-Ainfinity} is understood to be in effect. 

First, the higher operation $m_n$ is drawn as follows.

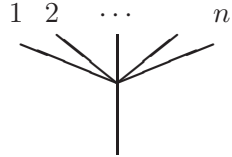
\setlength{\unitlength}{0.8cm}
\begin{picture}(12,4)
\thicklines
\put(7,2){\line(0,-1){1.2}}
\put(5.2,3){1}
\put(5.8,3){2}
\put(6.7,3){$\cdots$}
\put(8.6,3){$n$}
\put(7,2){\line(5,2){1.6}}
\put(7,2){\line(5,4){1}}
\put(7,2){\line(-5,2){1.6}}
\put(7,2){\line(-5,4){1}}
\put(7,2){\line(0,1){.8}}
\end{picture}

In this notation, the properties of Leibniz rule and associativity can be drawn as follows.

\setlength{\unitlength}{0.8cm}
\begin{picture}(15,3.5)
\thicklines
\put(1,1.5){\line(0,-1){1}}
\put(1,1.5){\line(5,6){.7}}
\put(1,1.5){\line(-5,6){.7}}
\put(2.5,1){=}
\put(0.8,0){$\partial$}
\put(3.1,2.6){$\partial$}
\put(4,1.5){\line(0,-1){1}}
\put(4,1.5){\line(5,6){.7}}
\put(4,1.5){\line(-5,6){.7}}
\put(5.1,1){+}
\put(7.2,2.5){$\partial$}
\put(6.5,1.5){\line(0,-1){1}}
\put(6.5,1.5){\line(5,6){.7}}
\put(6.5,1.5){\line(-5,6){.7}}
\put(10.255,2.4){\line(5,6){.7}}
\put(11,1.5){\line(0,-1){1}}
\put(11,1.5){\line(5,6){.7}}
\put(11,1.5){\line(-5,6){1.4}}
\put(12.3,1){=}
\put(14.55,2.4){\line(-5,6){.7}}
\put(13.8,1.5){\line(0,-1){1}}
\put(13.8,1.5){\line(5,6){1.4}}
\put(13.8,1.5){\line(-5,6){.7}}
\end{picture}

\noindent 
To justify the first equation, note that $\partial\otimes 1$ will produce no sign when applied to the input $x \otimes y$ since $|1|=0$, but $1\otimes \partial$ will have the sign $(-1)^{|x|}$.
\end{ntn} 

We now recall the Homotopy Transfer Theorem \cite{Kad-80} that allows one to transfer a dg algebra structure along a deformation retract yielding an $A_\infty$-structure on the retract. 
For this, we will define the descended higher operations using the tree notation introduced above. 

\begin{ipg}{\bf Homotopy Transfer Theorem for dg algebras.}
\label{HTT}
Let 
\[
(X,\partial^X) 
\stackrel[p]{i}{\longleftrightarrows}
(Y,\partial^Y)
\]
be a deformation retract with associated homotopy $h$ where $X$ is a dg algebra. 
As in Remark~\ref{A-dga}, one considers $X$ an $A_\infty$-algebra with $m_1$ equal to the differential, $m_2$ equal to the dg algebra product on $X$, and  $m_n^X=0$ for $n\geq 3$.

The Homotopy Transfer Theorem gives an $A_\infty$-structure on $Y$ as follows: 
First set $m_1^Y=\partial^Y$. 
For $n\geq 2$, the $n$th  operation $m_n^Y$ is defined as

\setlength{\unitlength}{0.8cm}
\begin{picture}(12,6)
\thicklines
\put(2,2){\line(0,-1){1.2}}
\put(2,2){\circle*{0.2}}
\put(0.2,3){1}
\put(.8,3){2}
\put(1.7,3){$\cdots$}
\put(3.6,3){$n$}
\put(2,2){\line(5,2){1.6}}
\put(2,2){\line(5,4){1}}
\put(2,2){\line(-5,2){1.6}}
\put(2,2){\line(-5,4){1}}
\put(2,2){\line(0,1){.8}}
\put(4,1.5){$\coloneqq $}
\put(5,1.5){$\mathlarger  {\mathlarger{\mathlarger{\mathlarger{\sum_{PBT_n} }}}}$}
\put(6.6,1.5){$\pm$}
\put(8.75,4){$i$}
\put(7.1,4){$i$}
\put(8,3){\line(5,-6){.4}}
\put(8.5,2.3){$h$}
\put(8.8,2.1){\line(5,-6){.4}}
\put(9.22,1.57){\line(5,6){.4}}
\put(9.22,1.57){\line(0,-1){.6}}
\put(8,3){\line(5,6){.7}}
\put(9.7,2.3){$h$}
\put(9.48,4){$i$}
\put(10,2.6){\line(5,6){.6}}
\put(8,3){\line(-5,6){.7}}
\put(10.3,3){\line(-5,6){.7}}
\put(10.6,3.5){$h$}
\put(11,3.85){\line(5,6){1}}
\put(11.4,4.32){\line(-5,6){.6}}
\put(12,5.2){$i$}
\put(10.6,5.2){$i$}
\put(9.2,.5){$p$}
\end{picture}

\noindent 
where the left hand side is the notation for $m_n^Y$ 
and where the sum is over $PBT_n$, the set of all planar binary rooted trees with $n$ leaves,
and the tree diagram pictured on the right is just a representative example of such a tree. 
The pattern of maps appearing on each tree is meant to indicate that every product is followed by an application of $h$, except for the last one, where instead $p$ is applied. 
The actual signs, indicated simply as $\pm$ above, are defined in the various sources quoted, but we shall not need them for our results. 
Again, in applying the maps in trees, the sign rule described in Definition~\ref{def-Ainfinity} is understood to be in effect. 

In particular, $m_2^Y$  and $m_3^Y$ are given by

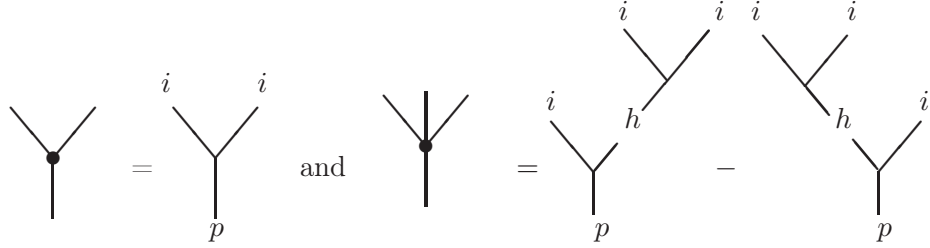
\setlength{\unitlength}{0.8cm}
\begin{picture}(15,5)
\thicklines
\put(0.5,1.8){\line(0,-1){1}}
\put(0.5,1.8){\circle*{0.2}}
\put(0.5,1.8){\line(5,6){.7}}
\put(0.5,1.8){\line(-5,6){.7}}
\put(1.8,1.5){=}
\put(3.9,2.9){$i$}
\put(2.3,2.9){$i$}
\put(3.2,1.8){\line(0,-1){1}}
\put(3.2,1.8){\line(5,6){.7}}
\put(3.2,1.8){\line(-5,6){.7}}
\put(3.1,.5){$p$}
\put(4.6,1.5){$\text{and}$}
\put(6.7,2){\line(0,-1){1}}
\put(6.7,2){\circle*{0.2}}
\put(6.7,2){\line(5,6){.7}}
\put(6.7,2){\line(-5,6){.7}}
\put(6.7,2){\line(0,1){.9}}
\put(8.2,1.5){$=$}
\put(10,2.3){$h$}
\put(9.48,1.57){\line(-5,6){.7}}
\put(9.48,1.57){\line(5,6){.4}}
\put(9.48,1.57){\line(0,-1){.7}}
\put(10.3,2.6){\line(5,6){1.1}}
\put(10.7,3.1){\line(-5,6){.7}}
\put(8.7,2.6){$i$}
\put(11.5,4.1){$i$}
\put(9.9,4.1){$i$}
\put(9.5,.5){$p$}
\put(11.5,1.5){$-$}
\put(13.7,4.1){$i$}
\put(12.1,4.1){$i$}
\put(13,3){\line(5,-6){.4}}
\put(13.5,2.3){$h$}
\put(13.8,2.1){\line(5,-6){.4}}
\put(14.22,1.57){\line(5,6){.7}}
\put(14.22,1.57){\line(0,-1){.7}}
\put(13,3){\line(5,6){.7}}
\put(14.9,2.6){$i$}
\put(13,3){\line(-5,6){.7}}
\put(14.2,.5){$p$}
\end{picture}

\noindent 
where the signs in the expression for $m_3^Y$ are the opposite of those in \cite{Vall-14} since, as we recall, his homotopy is the negative of ours. 
\end{ipg}

Under the hypotheses in this paper, we can show that the descent actually yields an $A_\infty$-algebra with all higher operations $m_{\geq 3}$ equal to zero. Proposition~\ref{descent-dga} does yield a dg algebra, and so one could extend it to an $A_\infty$-algebra by defining the higher operations $m_{\geq 3}$ equal to zero, but the Homotopy Transfer Theorem (\ref{HTT}) also gives a set of higher operations, which may not be the same. Here we prove that those vanish as well. 

\begin{prp}
\label{descent-dga-Ainfinity}
Let $X$ be a dg algebra. Consider a  deformation retract
\[
(X,\partial^X) 
\stackrel[p]{i}{\longleftrightarrows}
(Y,\partial^Y)
\]
with associated homotopy $h$ that satisfies the generalized Leibniz rule and $hi=0$. 
Then the $A_\infty$-algebra structure on $Y$ obtained from the dg algebra structure on $X$ via \ref{HTT} has trivial higher operations, that is, $m_n^Y=0$ for all $n\geq 3$.
\end{prp}

\begin{proof}
Recall from \ref{HTT} that the operations $m_n^Y$ for $n\geq 3$ descended from the dg algebra structure on $X$ are signed sums of elements described by planar binary rooted trees with $n$ leaves. 
The signs do not matter as we prove that every term equals zero. 
Indeed, each term always includes a factor of the form $h(i(\alpha)i(\beta))$ for some $\alpha, \beta \in Y$ (this will be nested inside other maps $i$, $h$ and $p$ and products from $X$). This vanishes as $h$ satisfies the generalized Leibniz rule and $hi=0$ holds. Hence these higher operations $m_n$ all vanish. 
\end{proof}

What if one begins with a complex $X$ that is an $A_\infty$-algebra rather than a dg algebra? 
First we recall the version of the Homotopy Transfer Theorem for this situation from \cite{Kad-82}, namely a more general version  that allows one to transfer an $A_\infty$-structure along a deformation retract yielding an $A_\infty$-structure on the retract. 
We again use the tree notation introduced above.

\begin{ipg}{\bf Homotopy Transfer Theorem for $A_\infty$-algebras.} 
\label{HTT-Ainfinity}
Let 
\[
(X,\partial^X) 
\stackrel[p]{i}{\longleftrightarrows}
(Y,\partial^Y)
\]
be a deformation retract with associated homotopy $h$ where $X$ is an $A_\infty$-algebra. 
The Homotopy Transfer Theorem gives an $A_\infty$-structure on $Y$ as follows: 
First set $m_1^Y=\partial^Y$. 
For $n\geq 2$, the operation $m_n^Y$ is defined as  

\setlength{\unitlength}{0.8cm}
\begin{picture}(12,6)
\thicklines
\put(2,2){\line(0,-1){1.2}}
\put(2,2){\circle*{0.2}}
\put(0.2,3){1}
\put(.8,3){2}
\put(1.7,3){$\cdots$}
\put(3.6,3){$n$}
\put(2,2){\line(5,2){1.6}}
\put(2,2){\line(5,4){1}}
\put(2,2){\line(-5,2){1.6}}
\put(2,2){\line(-5,4){1}}
\put(2,2){\line(0,1){.8}}
\put(4,1.5){$\coloneqq $}
\put(5,1.5){$\mathlarger  {\mathlarger{\mathlarger{\mathlarger{\sum_{PT_n} }}}}$}
\put(6.6,1.5){$\pm$}
\put(8.75,4){$i$}
\put(7.1,4){$i$}
\put(7.9,4){$i$}
\put(8,3){\line(5,-6){.4}}
\put(8.5,2.3){$h$}
\put(8.8,2.1){\line(5,-6){.4}}
\put(9.22,1.57){\line(5,6){.4}}
\put(9.22,1.57){\line(0,-1){.7}}
\put(8,3){\line(5,6){.7}}
\put(9.7,2.3){$h$}
\put(9.48,4){$i$}
\put(10,2.6){\line(5,6){.6}}
\put(8,3){\line(-5,6){.7}}
\put(8,3){\line(0,1){.8}}
\put(10.3,3){\line(-5,6){.7}}
\put(10.6,3.5){$h$}
\put(11,3.85){\line(5,6){.4}}
\put(11.4,4.32){\line(5,2){1.3}}
\put(11.4,4.32){\line(5,4){.7}}
\put(11.4,4.32){\line(-5,2){1.3}}
\put(11.4,4.32){\line(-5,4){.7}}
\put(11.4,4.32){\line(0,1){.6}}
\put(12,5.1){$i$}
\put(10.6,5.1){$i$}
\put(11.3,5.1){$i$}
\put(9.8,5.1){$i$}
\put(12.8,5.1){$i$}
\put(9.2,.5){$p$}
\end{picture}

\noindent 
where the left hand side is the notation for $m_n^Y$ 
and where the sum is over $PT_n$, the set of all planar (not necessarily binary) rooted trees with $n$ leaves,
and the tree diagram pictured on the right is just a representative example of such a tree, 
where the higher products are all occurring in $X$. 
Once again, the pattern is that every such product is followed by an application of $h$, except for the last one, where instead $p$ is applied. 
Again, the actual signs are defined in the various sources quoted, but we shall not need them for our results. 
\end{ipg}

Next we impose analogous but much stronger conditions on the homotopy when $X$ is merely an $A_\infty$-algebra, rather than a dg algebra. 
We do not have any example that satisfies this condition, but include this result for completeness in case it could be useful. 
We say that the homotopy $h$ on $X$ satisfies the \emph{generalized Leibniz rule for an $A_\infty$-algebra} if for every $n\geq 2$ one has 
\begin{equation*}
    h(m_n(a_1 \otimes \dots \otimes a_n)) \subseteq 
    \sum_{i=1}^n 
    m_n(X  \otimes \dots \otimes h(a_i) \otimes \dots \otimes X)
  \end{equation*}
where $h(a_i)$ is the $i$th factor and the other factors are $X$. 
Under this hypothesis, we can show that the formulas for the descended operations via the Homotopy Transfer Theorem for $A_\infty$-algebras (see \ref{HTT-Ainfinity}) are much simpler than usual (they are just the ones induced by going back and forth along the homotopy equivalence).

\begin{prp}
\label{descent-Ainfinity}
Let $X$ be an $A_\infty$-algebra with operations $m^X_n$ for $n\geq 1$.  
Consider a deformation retract
\[
(X,\partial^X) 
\stackrel[p]{i}{\longleftrightarrows}
(Y,\partial^Y)
\]
with associated homotopy $h$ that satisfies the generalized Leibniz rule for an $A_\infty$-algebra and $hi=0$. 
Then the $A_\infty$-algebra structure on $Y$ obtained from the $A_\infty$-algebra structure on $X$ via \ref{HTT-Ainfinity} has operations given by 
\[
m_n^Y=p\,m_n^X(i\otimes  \dots\otimes i)
\] 
for all $n\geq 1$.
\end{prp}

\begin{proof}
Note that in \ref{HTT-Ainfinity}, it follows from the construction and the properties of a deformation retract that 
\[
m_1^Y=\partial^Y=p\,m_1^X \,i
~{\text{ \ and \ }}~
m_2^Y=p \, m_2^X (i\otimes i).
\] 
This covers the cases $n=1,2$. 

Recall from \ref{HTT-Ainfinity} that the operations $m_n^Y$ for $n\geq 3$ descended from the dg algebra structure on $X$ are signed sums of elements described by planar rooted trees with $n$ leaves. 
The signs do not matter as all of the terms vanish except one. Indeed, expanding using that $hi=0$ and the generalized Leibniz rule for an $A_\infty$-algebra leaves only the desired term as that is the only one given by a tree with only one (higher) operation, hence simply followed by an application of $p$ and not involving the homotopy $h$; this term is known to be positive. 
\end{proof}

 \section{Transfer of dg algebra structures and the Perturbation Lemma}
 \label{sec-perturbation}

The second main aim of this paper is to use descent along a deformation retract to find a dg algebra structure on a well known complex, which we do in the next section. 
Building this retract involves a homological tool called the Perturbation Lemma. 
In this section we extend the descent result in Proposition~\ref{descent-dga} to perturbations of the original setting, resulting in Proposition~\ref{descent-dga-perturbation}. 
One can similarly extend Propositions~\ref{descent-dga-Ainfinity} and \ref{descent-Ainfinity}; see Remark~\ref{analogous}. 

The Perturbation Lemma generates new homotopy equivalences from initial ones; in general the aim is to modify the differentials of the complexes while maintaining a homotopy equivalence. 
For more details, the reader may consult Crainic's exposition in \cite{Cra-04} and also \cite{DyMu-13} where Dyckerhoff and Murfet develop the lemma for the analogous case of matrix factorizations. 
The Perturbation Lemma is especially useful for double complexes where one can temporarily forget either the horizontal or the vertical differentials and add them back in later as the ``perturbation"; this is the context in which we will apply it in Section~\ref{sec-L}. 

We define some terminology we use in stating the Perturbation Lemma. 

\begin{dfn}
\label{def-perturbation}
Consider a set of homotopy equivalence data 
\begin{equation*}
   (X,\partial^X) 
\stackrel[p]{i}{\longleftrightarrows}
(Y,\partial^Y)
\end{equation*}
with associated homotopy $h$. 
A \emph{perturbation} is a map $\delta$ on $X$ of the same degree as the differential $\partial^X$ such that $(\partial^X+\delta)^2=0$, that is,  $\partial^X+\delta$ is again a differential. 
The perturbation $\delta$ is called {\it small} if $1-\delta h$ is invertible. 
Most commonly, this happens when $\delta h$ is elementwise nilpotent  for then one has 
\[
(1-\delta h)^{-1} 
= \sum_{j=0}^\infty (\delta h)^j 
= 1+(\delta h)+(\delta h)^2+\cdots
\]
where the sum is finite on each element of $X$.  
\end{dfn}

\begin{dfn}\label{def-infinity}
Consider a set of homotopy equivalence data 
\begin{equation*}
   (X,\partial^X) 
\stackrel[p]{i}{\longleftrightarrows}
(Y,\partial^Y)
\end{equation*}
with associated homotopy $h$. Let $\delta$ be a small perturbation on $X$, and 
let  $A=(1-\delta h)^{-1} \delta$. We define the following new data 
\begin{equation*}
   (X,\partial_\infty^X) 
\stackrel[p_\infty]{i_\infty}{\longleftrightarrows} (Y,\partial_\infty^Y)
\end{equation*}
where 
\begin{equation*}
   i_\infty=i+hAi, ~ p_\infty=p+pAh, 
   ~ \partial_\infty^X=\partial^X+\delta,
   \text{ and}~ \partial_\infty^Y=\partial^Y+pAi
\end{equation*}
and set 
\begin{equation*}
   h_\infty=h+hAh  
\end{equation*}
Note that when $\delta h$ is elementwise nilpotent, then the formulas can be rewritten as follows. 
\begin{align*}
i_\infty&=(1+(h\delta)+(h\delta)^2+\cdots)i
\\
p_\infty&=p(1+(\delta h)+(\delta h)^2+\cdots)
\\
h_\infty&=h(1+(\delta h)+(\delta h)^2+\cdots)
\\
\partial^Y_\infty
&=\partial^Y+p\delta i_\infty 
\\
&=\partial^Y+p_\infty\delta i
\end{align*}
\end{dfn}

\begin{dfn}
\label{special}
A \emph{special deformation retract} is a deformation retract that satisfies the following equations 
\begin{equation}
\label{specialDR}
hi=0, ~
ph=0, ~
h^2=0. 
\end{equation}
These ensure that the property $pi=1$ is inherited by the perturbed data.

As described in \cite{Cra-04}, any deformation retract can be converted into a special one by modifying the chosen homotopy in several steps; the drawback is that when seeking explicit formulas the resulting maps  become more complicated. Fortunately, for our application, the deformation retracts involved are all special. 
\end{dfn}

With this terminology, we are now ready to state the  Perturbation Lemma. 

\begin{ipg}{\bf Perturbation Lemma.}
\label{perturbation}
Given a set of homotopy equivalence data
\begin{equation*}
   (X,\partial^X) 
\stackrel[p]{i}{\longleftrightarrows}
(Y,\partial^Y)
\end{equation*}
 with associated homotopy  $h$, its perturbation via a small perturbation gives a set of homotopy equivalence data
\begin{equation*}
   (X,\partial_\infty^X)  
\stackrel[p_\infty]{i_\infty}{\longleftrightarrows} (Y,\partial_\infty^Y)
\end{equation*}
with associated homotopy $h_\infty$. 

If, furthermore, the original homotopy equivalence  is a special deformation retract then so is the resulting one, that is, 
\[
p_\infty i_\infty =1, \ 
h_\infty ^2=0, \ 
~  h_\infty i_\infty =0, \ 
~ {\text{and}} ~
p_\infty h_\infty =0, 
\]
\end{ipg}

Recall that by Proposition~\ref{descent-dga}, given a special  deformation retract whose homotopy satisfies the generalized Leibniz rule, a dg algebra structure can be transferred along it. 
One might want to use the Perturbation Lemma to obtain new deformation retracts to which one could apply this proposition. However, even if the original homotopy satisfies the generalized Leibniz rule, the new one may no longer satisfy it. 
We remedy this by proving an extension of the descent results as follows. 

\begin{prp} 
\label{descent-dga-perturbation}
Let $X$ be a dg algebra. 
Consider a special deformation retract
\[
(X,\partial^X) 
\stackrel[p]{i}{\longleftrightarrows}
(Y,\partial^Y)
\]
with associated homotopy $h$ that satisfies the generalized Leibniz rule, and 
let $\delta$ be a small perturbation on $X$. 

If the perturbed complex $(X,\partial^X_\infty)$ from Lemma~\ref{perturbation} remains a dg algebra via the same product (equivalently, if $\delta$ satisfies the Leibniz rule), then the following product defines a dg algebra structure on the perturbed complex $(Y,\partial^Y_\infty)$
\[
\alpha  \beta 
\stackrel{\text{ def }}{=} 
p_\infty \left( i_\infty(\alpha)  i_\infty(\beta) \right)
{\text{ for }}~
\alpha, \beta \in Y
\]
where the product inside parentheses is the one in $X$. 

Moreover, with this structure on $Y$, the map $i_\infty$ becomes a dg algebra homomorphism.
\end{prp}

\begin{proof}
The proof is the same as that of Proposition~\ref{descent-dga} with the exception that to prove associativity and that $i_\infty$ is a dg algebra homomorphism, one needs to show that 
\[
(\partial^X_\infty h_\infty+h_\infty \partial^X_\infty) \left ( i_\infty (\alpha)  i_\infty(\beta)\right )
=0.
\] 
Here this follows by similar reasoning due to the facts that one has 
\begin{equation*}
   h_\infty=h+hAh ~{\text{ and }} ~
   i_\infty=i+hAi
   ~{\text{ where  }}~
   A=(1-\delta h)^{-1} \delta
   \end{equation*}
and that $h^2=0$ and $hi=0$. 
\end{proof}

\begin{rmk}
\label{analogous}
We note that the analogous generalizations of Propositions~\ref{descent-dga-Ainfinity} and \ref{descent-Ainfinity} hold as well, with similar proofs modified as the one above. 
\end{rmk}

 \section{Application to a minimal resolution}
 \label{sec-L}

Let $R=k[x_1, \cdots, x_n]$ be a polynomial ring over a field $k$.  
In \cite{BuchEis-75}, Buchsbaum and Eisenbud introduced
the minimal free resolution $\mathbb L_a$ of the quotient $R/(x_1, \cdots, x_n)^a$ of $R$ by powers of the homogeneous maximal ideal. 
In \cite{Srin89}, Srinivasan gives a dg algebra structure on $\mathbb L_a$ using Young tableaux. 
In this section, we use the Perturbation Lemma in a simple way to obtain a dg algebra structure on $\mathbb L_a$ that is $S_n$-invariant. Our approach works in characteristic zero and in positive characteristic provided that the characteristic is large enough.

We begin by recalling the definition of the resolution $\mathbb L_a$ and relating it to (the totalization of) a truncation of a certain double complex in (\ref{chunk-tautological}), (\ref{chunk-L}), and (\ref{chunk-truncation}). In (\ref{truncated-homotopy}) and  (\ref{chunk-PL}) we use the Perturbation Lemma to form a deformation retract between them, as long as one has an appropriately nice associated homotopy so that one can apply Proposition~\ref{descent-dga-perturbation}. Lastly we define such a homotopy using a scaled de Rham differential in (\ref{sigmadef}), proving its properties in Lemmas~\ref{sigmasquared} and ~\ref{scaledLeibniz}, culminating in Theorem~\ref{thm-L}. 

\begin{ipg} 
\label{chunk-tautological}
Here we define the double complex we will be working with. 
This is simply a rearrangement of a Koszul complex as a  double complex of free $R$-modules; see Remark~\ref{rmk-tautological}. 

Let $S=R[y_1,\dots, y_n]$ be a polynomial ring and   
let $\Lambda=R \langle e_1,\dots, e_n \rangle$ be an exterior algebra. Consider the following (anticommutative) double complex whose rows are the strands of the Koszul complex $K(y_1,\dots, y_n;S)$ and whose columns are the tensor product over $R$ of the Koszul complex over $R$ on $x_1, \dots, x_n$ with the graded pieces $S_a$ of $S$. 
We denote both it and its totalization by $\mathbb S$, as it is clear everywhere from the context which we mean. 
All the tensor products in the diagram are over $R$. 

\begin{equation}
\label{diagramS}
\begin{aligned} 
    \xymatrixrowsep{.8pc} \xymatrixcolsep{1pc}
    \xymatrix{    
      & 
      & 
      & 
      &  
      & 
      & \Lambda^{n} \otimes S_{a} \ar[d]_d \ar[r]^{\ \ \ \kappa}
      & \cdots 
\\
      & 
      & 
      & 
      &  
      & \cdots \ar[r]^{\! \! \! \! \!\!\! \! \! \! \!\!\!  \kappa}
      & \Lambda^{n-1} \otimes S_{a} \ar[d] \ar[r]^{\ \ \ \ \kappa}
      & \cdots 
\\
      & 
      & \vdots \ar[d]_d
      & 
      &  
      & 
      & \vdots\ar[d]_d 
      & 
\\
      & \vdots \ar[d]_d
      & \Lambda^{a} \otimes S_{2} \ar[d]_d \ar[r]^{\ \ \  \kappa}
      & \cdots
      &  
      & \cdots \ar[r]^{\! \! \! \! \!\!\! \kappa} 
      & \Lambda^2 \otimes S_{a} \ar[d]_d\ar[r]^{\ \ \ \kappa}  
      & \cdots 
\\
      \vdots \ar[d]_d
      & \Lambda^{a} \otimes S_{1} \ar[d]_d \ar[r]^{\!\!\!\!\kappa}
      & \Lambda^{a-1} \otimes S_{2} \ar[d]_d \ar[r]^{\ \ \ \ \ \kappa}
      & \cdots
      &  
      & \cdots \ar[r]^{\!\!\!\!\!\!\!\kappa}
      & \Lambda^1 \otimes S_{a} \ar[d]_d\ar[r]^{\ \ \ \kappa}
      & \cdots
\\
       \Lambda^{a} \otimes S_{0} \ar[d] \ar[r]^{\!\!\!\!\kappa} 
      & \Lambda^{a-1} \otimes S_{1} \ar[d] \ar[r]^{\kappa} 
      & \Lambda^{a-2} \otimes S_{2} \ar[d] \ar[r]^{\ \ \ \ \ \kappa} 
      & \cdots
      &  
      & \cdots \ar[r]^{\!\!\!\!\!\!\!\!\kappa}
      & \Lambda^0 \otimes S_{a} 
      & 
\\
       \vdots \ar[d]_d 
      & \vdots\ar[d]_d
      & \vdots\ar[d]_d
      &  
      &  
      & 
      &
      & 
\\
        \Lambda^3 \otimes S_{0} \ar[r]^{ \kappa}\ar[d]_d
      & \Lambda^2 \otimes S_{1} \ar[r]^{\kappa}\ar[d]_d
      & \Lambda^1 \otimes S_{2} \ar[d]_d\ar[r]^{\ \ \ \ \kappa}
      & \cdots 
      &  
      & 
      & 
      & 
\\
        \Lambda^2 \otimes S_{0} \ar[r]^{\kappa}\ar[d]_d
      & \Lambda^1 \otimes S_{1} \ar[r]^{\kappa}\ar[d]_d
      & \Lambda^0 \otimes S_{2} 
      &  
      &  
      & 
      & 
      & 
\\
        \Lambda^1 \otimes S_{0} \ar[r]^{\kappa}\ar[d]_d
      & \Lambda^0 \otimes S_{1} 
      & 
      & 
      &  
      & 
      & 
      & 
\\
        \Lambda^0 \otimes S_{0} 
      & 
      &   
      & 
      &  
      & 
      & 
      & 
\\
}
\end{aligned}  
  \end{equation}

More explicitly, the horizontal differentials $\kappa_{i,a}: \Lambda^i\otimes S_a \to \Lambda^{i-1} \otimes S_{a+1}$ are given by
\begin{equation}\label{kappadef}
    \kappa_{i,a}(e_{i_1} \wedge\cdots\wedge e_{i_i} \otimes p) = 
\sum_{j=1}^a (-1)^j e_{i_1} \wedge\cdots\wedge \widehat{e}_{i_j}\wedge\dots\wedge e_{i_i}\otimes y_{i_j}  p
\end{equation}
and the vertical differentials $d_{i,a}\colon \Lambda^i\otimes S_a \to \Lambda^{i-1}\otimes S_a$ are given by 
\begin{equation*}
d_{i,a}=\kos_i \otimes 1
\end{equation*}
where $\kos_i$ denotes the $i$th differential in the Koszul complex $K(x_1, \dots, x_n;R)$. 

For the totalization of this double complex (or of truncations of it), since it is anticommutative,  the differentials are defined as
\[
\partial_i = \sum_a (\kappa_{i,a} + d_{i,a})
\]
without adding any signs. For simplicity we write
\[
\partial=\kappa + d.
\]
We continue to omit the indices on the maps when there is no ambiguity. 
\end{ipg} 

\begin{rmk}
\label{rmk-tautological} 
As an aside, we give a slightly different way of obtaining a double complex which could have been used in this section. It differs only in signs from the one pictured in (\ref{diagramS}), but comes from a well known construction.

Let $V$ be a $k$-vector space with $\dim_k V = n$, and consider the symmetric and exterior algebras
\begin{align*}
\overline{S}&=S(V) 
\cong k[x_1', \dots, x_n'] 
\cong k[x_1'', \dots, x_n''] 
\\[2mm]
\overline{\Lambda}&=\Lambda(V) 
\cong k \langle e_1, \dots, e_n \rangle
\end{align*}
Consider $\overline{S}\cong R$ as a module over its enveloping algebra $\overline{S}^{\,e}=\overline{S}\otimes_k \overline{S}$ via the multiplication map. 
Its  minimal graded free resolution, after identifying the two copies of $\overline{S}$ with polynomial rings as in the display above, is the Koszul complex $\overline{\Lambda} \otimes_k \overline{S} \otimes_k \overline{S} $ on the regular sequence $\{ x_i' \otimes 1 - 1 \otimes x_i''   \}$. 
Rearranging factors, it can be expressed as 
\[
\overline{S} \otimes_k \overline{\Lambda} \otimes_k \overline{S} 
=
\underbrace{\overline{S} \otimes_k \overline{\Lambda}}_{\partial'} \otimes_k \overline{S} 
=
\overline{S} \otimes_k \underbrace{\overline{\Lambda} \otimes_k \overline{S}}_{\partial''}
\]
with the homological degree being the degree of the middle factor and 
\[
\partial = \partial' \otimes 1 - 1 \otimes \partial'' = d-\kappa
\] 
where  
$\partial'$ is the Koszul differential on $x_1', \dots, x_n'$ 
and 
$\partial''$ is the Koszul differential on $x_1'', \dots, x_n''$. 
Viewing graded strands, one can write this as a totalization of an anticommutative double complex of free $R$-modules given by 
\[ 
\overline{S} \otimes_k \overline{\Lambda}^i \otimes_k \overline{S}_j
\cong 
R \otimes_k \overline{\Lambda}^i \otimes_k \overline{S}_j
\cong 
(R \otimes_k \overline{\Lambda}^i) \otimes_R (R \otimes_k \overline{S}_j)
\cong 
\Lambda^i \otimes_R S_j
\]
Although this double complex differs from the one pictured in (\ref{diagramS}) by a sign on the horizontal maps $\kappa$, one could equally well use this complex in the rest of this section; similarly, one could obtain the double complex (\ref{diagramS}) from a Koszul complex by using $-x_i''$ in place of $x_i''$ above; note that it would no longer be a resolution of $R$ over its enveloping algebra.
\end{rmk}


\begin{ipg} 
\label{chunk-L}
We introduce the complexes $\mathbb L_a$ of Buchsbaum and Eisenbud here. 
They show that this is a minimal $R$-free resolution of $R/\fm^a$, where $\fm$ is the homogeneous maximal ideal of $R$.  

It is well known that the rows of the double complex (\ref{diagramS}) except the bottom one are exact; in fact, they can be viewed as the result of applying a base change to the strands of the tautological Koszul complex (see, for example, \cite{MiRa-18}). 
Hence they are contractible as they consist of free $R$-modules. 
So one can define free $R$-modules
\[
L_{i,a}=
\im \kappa_{i+1,a-1}
=
\ker \kappa_{i,a}
\cong 
\coker \kappa_{i+2,a-2},
\]
in other words, with split exact sequences
\begin{align*}
\Lambda^{i+2}\otimes S_{a-2}
\xra{\kappa_{i+2,a-2}}
\Lambda^{i+1}&\otimes S_{a-1}
\xra{\kappa_{i+1,a-1}}
L_{i,a} \lra 0
\\[3mm]
0 \lra L_{i,a}
\xra{\subseteq}
\Lambda^{i} &\otimes S_{a}
\xra{\kappa_{i,a}}
\Lambda^{i-1} \otimes S_{a+1}.
\end{align*}
The vertical differentials $d$ in the diagram induce maps on these modules, which we again denote by $d$, to yield a complex 
\begin{equation*}
    \mathbb L_a \colon  0\to L_{n-1,a}\xra{d_{n-1}}
L_{n-2,a}\xra{d_{n-2}} \dots \to
L_{0,a} \xra{\varepsilon} R \to 0
\end{equation*}
augmented by the evaluation map 
\begin{equation}
\label{epsilon}
\varepsilon\colon L_{0,a}=\Lambda^0\otimes S_a\cong S_a \to R
\end{equation}
induced by the evaluation map from $S=R[y_1, \dots, y_n]$ to $R=k[x_1, \dots, x_n]$ sending $y_i$ to $x_i$.
\end{ipg} 

\begin{ipg} 
\label{chunk-truncation}
Next we define $\tr_{\geq a}(\mathbb S)$ and $ \tr_{\leq a-1}(\mathbb S)$ to be the totalizations of  the truncations  at column $a$ of the anticommutative double complex $\mathbb S$ 
\begin{align*}
    \{ \Lambda^i \otimes S_j \mid j \geq a \}~ \text{and}~
    \{ \Lambda^i \otimes S_j \mid j \leq a-1 \},
\end{align*}
respectively, with differentials inherited from $\mathbb S$.
It is well-known that there is a quasi-isomorphism, and hence a homotopy equivalence, 
\[
\tr_{\leq a-1}(\mathbb S)
\simeq 
\mathbb L_a
\] 
but we will re-derive this via the Perturbation Lemma in order to simultaneously transfer a dg algebra structure from $\tr_{\leq a-1}(\mathbb S)$ over to $\mathbb L_a$ (by obtaining a special deformation retract rather than just any homotopy equivalence).

To set up for this, we first argue as in  \cite{Wut-04} that the left truncation $\tr_{\leq a-1}(\mathbb S)$ itself has a natural dg algebra structure. Indeed, the entire complex $\mathbb S$ is a dg algebra with the obvious  multiplication: 
for $\alpha\in\Lambda^i\otimes S_a$ and $\beta\in\Lambda^j\otimes S_b$, the product is obtained by multiplying the factors in $\Lambda$ and in $S$ independently.
It satisfies the Leibniz rule and other properties of a dg algebra because the differentials $\kappa$ and $d$ do and because homological degree in the totalization of $\mathbb S$ is, in fact, given by the degree in $\Lambda$. 
With this multiplication, the right truncation $\tr_{\geq a}(\mathbb S)$ is clearly a dg ideal and the quotient complex 
\[
\mathbb X_a 
\stackrel{\text{ def }}{=}
\tr_{\leq a-1}(\mathbb S)
\cong 
\mathbb S / \tr_{\geq a}(\mathbb S) 
\]
is therefore a dg algebra. Concretely, the resulting product on the left truncation $\tr_{\leq a-1}(\mathbb S)$ is given by the multiplication on $\mathbb S$ with the proviso that any terms landing in $\Lambda^i\otimes S_j$ with $j \geq a$ are taken to be zero. 
\end{ipg} 

For the next step, we first need a tool for converting a split exact sequence to a deformation retract from any truncation to the image of the differential at the truncation.  

\begin{ipg} 
\label{truncated-homotopy}
Let $(X,\partial^X)$ be a contractible complex of $R$-modules, i.e., one that is homotopy equivalent to zero via a homotopy $s$, (i.e., one that is split exact). 
Denote its truncation at position $c$ by 
\[
\tr_{\geq c}(X) = 
\cdots  
\lra 
X_n 
\xra{\partial_{n+1}^X}
\cdots 
\lra X_{c+1} 
\xra{\partial_{c+1}^X} X_c 
\lra 0
\]
Let $\im \partial^X_c$ denote the stalk complex with this module in degree $c$ and 0 modules elsewhere. The chain maps $i$ and $p$ given in degree $c$ by $s_{c-1}$ and $\partial_c^X$, respectively, yield a deformation retract
\begin{align*}
    \tr_{\geq c}(X)
\stackrel[p]{i}{\longleftrightarrows}
\im \partial^X_c
\end{align*}
with associated homotopy $h=-s|_{\tr_{\geq c}X}$.   Indeed one can easily check that $pi=1$ and $ip\simeq 1$ via the homotopy $h$. 
Note that one could also use $\coker \partial^X_{c+1}$ instead of $\im \partial^X_c$ with appropriate $i$ and $p$. 
If, furthermore, the original contracting homotopy satisfies $s^2= 0$, then the deformation retract is special:  $h^2=0$ and $hi=0$ and one always has  $ph=0$  due to the fact that $p=0$ in degrees $n\neq c$.
\end{ipg} 

Next we want to transfer this structure from $\mathbb X_a = \tr_{\leq a-1}(\mathbb S)$, which has a dg algebra structure by (\ref{chunk-truncation}) to the minimal free resolution $\mathbb L_a$ of $R/\fm^a$. 
By Proposition~\ref{descent-dga-perturbation}, it suffices to find a special deformation retract of the form 
\[
(\mathbb X_a,\partial^{\mathbb X_a}) 
\stackrel[]{}{\longleftrightarrows}
(\mathbb L_a,d)
\]
that is a perturbation of a special deformation retract whose homotopy satisfies generalized Leibniz. 
Note that the differential of $\mathbb X_a$ is exactly $\kappa+d$.
In \ref{chunk-PL}, we  discuss how one can find this deformation retract,
given a contracting homotopy on the higher rows of $\mathbb S$, which we define in \ref{sigmadef} and whose required properties we establish in  Lemmas~\ref{sigmasquared} and \ref{scaledLeibniz}. 

\begin{ipg} 
\label{chunk-PL}
Here is overview of how we obtain such a deformation retract using the Perturbation Lemma; see (\ref{perturbation}). 
First we form a deformation retract between two complexes $\mathbb X_a^\circ$ and $\mathbb L_a^\circ$, where $\mathbb X_a^\circ$ is obtained from $\mathbb X_a$ by replacing the vertical differentials $d$ by 0 and $\mathbb L_a^\circ$ is the complex $\mathbb L_a$ with differentials set equal to zero. 
We do this via (\ref{truncated-homotopy}) using the homotopy from \ref{sigmadef}. 
Second we use the Perturbation Lemma to reinsert the original differentials on each, which has the effect of modifying the maps $i$ and $p$.

We start by finding a deformation retract of the form 
\[
(\mathbb X_a^\circ,\kappa) 
\stackrel[p]{i}{\longleftrightarrows}
(\mathbb L_a^\circ,0)
\]
with a homotopy $h$. 
For rows of $\mathbb X_a^\circ$ except the bottom one, 
we use (\ref{truncated-homotopy}) as follows.  Recall that the rows of $\mathbb S$ are split exact with a contracting homotopy that we call $\sigma$ (an explicit one is given in Remark~\ref{sigmadef}). 
So each row that gets truncated has a deformation retract onto the image $L_{i,a}$ of the next horizontal differential; see diagram (\ref{diagramX0L0}). 
Note that some of the lower rows will remain intact and hence are homotopy equivalent to zero; see (\ref{diagramX0L0}). 
On the other hand, the row $\Lambda^0\otimes S_0$ at the bottom of the diagram  is not exact and so needs to be dealt with separately in conjunction with $R=(\mathbb L_a)_0$. 
For this we use that there is an isomorphism $\varepsilon\colon \Lambda^0\otimes S_0 \to R$ defined in (\ref{epsilon}). 

Putting this all together, one obtains chain maps given by
\begin{align*}
i &= 
\begin{cases}
\sigma & 
{\text{on }} L_{i,a} 
\\
\varepsilon^{-1} &
{\text{on }} R
\end{cases}\\
p &= 
\begin{cases}
\kappa & 
{\textrm{on }} 
\Lambda^i \otimes S_{a-1}
{\textrm{ for }} 
i>0 
\\
\varepsilon & 
{\textrm{on }} 
\Lambda^0 \otimes S_0 
\\
0 & 
{\textrm{else,}} 
\end{cases}
\end{align*}
with the property that $pi=1$ and $ip\simeq 1$ via the homotopy $h=-\sigma\vert_{\mathbb X_a^\circ}$. 
The maps $i$ and $p$ are pictured in following diagram.
\begin{equation}
\label{diagramX0L0}
\begin{aligned} 
    \xymatrixrowsep{1.3pc} \xymatrixcolsep{0.6pc}
    \xymatrix{    
      & 
      & 
      & 
      & 
      &
      & 0
\\
      & 
      & 
      & 
      & \Lambda^{n} \otimes S_{a-1} 
      \ar@/_{.9pc}/[rr]^{\ \ p}
      &
      & \ar@/_{.9pc}/[ll]_{\ \ i}L_{n-1,a}  
\\
      & 
      & 
      & \ \vspace{10mm} \cdots \ar[r]^{}
      & \Lambda^{n-1} \otimes S_{a-1} \ar@/_{.9pc}/[rr]^{\ \ p}
      &
      & \ar@/_{1pc}/[ll]_{\ \ i}L_{n-2,a} 
\\
      & \vdots 
      & 
      &  
      & \vdots 
      &
      & \vdots
\\
      \vdots 
      & \Lambda^{a} \otimes S_{1}  \ar[r]^{}
      & \cdots 
      & \cdots \ar[r]^{}
      & \Lambda^2 \otimes S_{a-1} \ar@/_{.9pc}/[rr]^{\ \ p}
      &
      &\ar@/_{1pc}/[ll]_{\ \ i} L_{1,a}  
\\
       \Lambda^{a} \otimes S_{0}  \ar[r]^{} 
      & \Lambda^{a-1} \otimes S_{1}  \ar[r]^{}
      & \cdots 
      & \cdots \ar[r]^{}
      & \Lambda^1 \otimes S_{a-1} \ar@/_{.9pc}/[rr]^{\ \ p}
      &
      & \ar@/_{1pc}/[ll]_{\ \ i}L_{0,a}  
\\
       \Lambda^{a-1} \otimes S_{0} 
      & \vdots
      &  
      & \cdots \ar[r]^{}
      & \Lambda^0 \otimes S_{a-1} 
      &
      & R\ar@/^{3pc}/[lllllldddd]_{\ \ i}
\\
        \vdots 
      & \Lambda^2 \otimes S_{1} \ar[r]^{}
      & \cdots 
      & 
      & 
      &
      & 
\\
        \Lambda^2 \otimes S_{0} \ar[r]^{}
      & \Lambda^1 \otimes S_{1} \ar[r]^{}
      & \cdots 
      & 
      & 
      &
      & 
\\
        \Lambda^1 \otimes S_{0} \ar[r]^{}
      & \Lambda^0 \otimes S_{1}
      & 
      & 
      & 
      &
      & 
\\
        \Lambda^0 \otimes S_{0} \ar@/_3pc/[uuuurrrrrr]_{p}
      & 
      & 
      & 
      & 
      &
      & 
\\
}
\end{aligned}  
  \end{equation}

Next we apply the Pertubation Lemma, adding the missing vertical differentials $d$ of $\mathbb X_a$ and $d$ of $\mathbb L_a$. 
More precisely, consider the perturbation $\delta=d$ on $\mathbb X^\circ_a$; this is a small perturbation since the double complex $\mathbb X_a$ is bounded. 
First, we check that that the differentials on $\mathbb L_a$ obtained in this way are the original differentials on $\mathbb L_a$. 
This is because one has 
\[
\partial^{\mathbb L_a^\circ}+p_\infty\delta i
=
0+p(1+(dh)+(dh)^2+\cdots)di
=pdi=dpi=d
\]
where the second equality follows from the fact that $p$ vanishes on most of the diagram, the third one follows from the fact that $p$ is defined using $\kappa$ and $\varepsilon$, as well as the commutativity of diagram (\ref{diagramS}) and the  properties of $\varepsilon$, and the last one is because $pi=1$. 

In summary, one gets a homotopy equivalence
\begin{equation}\label{xlhe}
(\mathbb X_a,\partial^{\mathbb X_a}\!\!=\kappa+d) 
\stackrel[p_\infty]{i_\infty}{\longleftrightarrows}
(\mathbb L_a,\partial^{\mathbb L_a}=d)
\end{equation}

For later use, we calculate the new chain maps $i_\infty$ and $p_\infty$, as well as the associated homotopy $h_\infty$ using the formulas in Definition~\ref{def-infinity}. 
The map $i_\infty$ is given by 
\[
i_\infty=(1+(h \delta)+(h \delta)^2+\cdots)i
\]
where $\delta=d$, and this can be written as
\begin{equation}
\label{i-infty}
i_\infty = 
\begin{cases}
(1+(-\sigma d)+(-\sigma d)^2+\cdots) \sigma & 
{\text{ on }} L_{i,a} \\
\varepsilon^{-1} & 
{\text{ on }} R. 
\end{cases}
\end{equation}
In contrast, the map $p_\infty$ is remarkably simpler since $p$ equals zero on most of its domain. Indeed it is given by   
\[
p_\infty=p(1+(\delta h)+(\delta h)^2+\cdots)
\]
which can be written as
\begin{equation}
\label{p-infty}
p_\infty = 
\begin{cases}
\kappa & 
{\textrm{on }} 
\Lambda^i \otimes S_{a-1}
{\textrm{ for }} 
i>0
\\
\varepsilon & 
{\textrm{on }} 
\Lambda^0 \otimes S_j 
{\textrm{ for all }} j
\\
0 & 
{\textrm{else.}} 
\end{cases}
\end{equation}
We record also the resulting homotopy for $i_\infty p_\infty \simeq 1$, which is 

\begin{align}
\label{h-infty}
\begin{aligned}
h_\infty&=h(1+(\delta h)+(\delta h)^2+\cdots)
\\
&=-\sigma(1+(-d \sigma)+(-d \sigma)^2+\cdots) 
\end{aligned}
\end{align}

The map $p_\infty$ has the form pictured in the following diagram. 

\begin{equation}
\label{diagramXL}
\begin{aligned} 
    \xymatrixrowsep{0.6pc} \xymatrixcolsep{0.6pc}
    \xymatrix{    
      & 
      & 
      & 
      & 
      & 0 \ar[d] 
      &  
\\
      & 
      & 
      & 
      & \Lambda^{n} \otimes S_{a-1} \ar[d] \ar[r]^{\ \ p_\infty}
      & L_{n-1,a} \ar[d] \ar[r]^{}
      & 0 
\\
      & 
      & 
      & \ \vspace{10mm} \cdots \ar[r]^{}
      & \Lambda^{n-1} \otimes S_{a-1} \ar[d] \ar[r]^{\ \ p_\infty}
      & L_{n-2,a} \ar[d] \ar[r]^{}
      & 0
\\
      & \vdots \ar[d]
      & 
      &  
      & \vdots\ar[d] 
      & \vdots \ar[d]  
      & 
\\
      \vdots \ar[d]
      & \Lambda^{a} \otimes S_{1} \ar[d] \ar[r]^{}
      & \cdots
      & \cdots \ar[r]^{}
      & \Lambda^2 \otimes S_{a-1} \ar[r]^{\ \ p_\infty}\ar[d] 
      & L_{1,a} \ar[d]\ar[r]^{}
      & 0
\\
       \Lambda^{a} \otimes S_{0} \ar[d] \ar[r]^{} 
      & \Lambda^{a-1} \otimes S_{1} \ar[d] \ar[r]^{}
      & \cdots 
      & \cdots \ar[r]^{}
      & \Lambda^1 \otimes S_{a-1} \ar[r]^{\ \ p_\infty}\ar[d] 
      & L_{0,a} \ar[d]\ar[r]^{}
      & 0
\\
       \Lambda^{a-1} \otimes S_{0} \ar[d] 
      & \vdots\ar[d]
      &  
      & \cdots \ar[r]^{}
      & \Lambda^0 \otimes S_{a-1} \ar[r]^{\ \ p_\infty}
      & R
      & 
\\
        \vdots\ar[d]
      & \Lambda^2 \otimes S_{1} \ar[r]^{}\ar[d]
      & \cdots 
      & 
      & 
      & 
      & 
\\
        \Lambda^2 \otimes S_{0} \ar[r]^{}\ar[d]
      & \Lambda^1 \otimes S_{1} \ar[r]^{}\ar[d]
      & \cdots 
      & 
      & 
      & 
      & 
\\
        \Lambda^1 \otimes S_{0} \ar[r]^{}\ar[d]
      & \Lambda^0 \otimes S_{1} \ar@/_2pc/[uuurrrr]^{p_\infty}
      & 
      & 
      & 
      & 
      & 
\\
        \Lambda^0 \otimes S_{0} \ar@/_4pc/[uuuurrrrr]_{p_\infty}
      & 
      & 
      & 
      & 
      & 
      & 
\\
}
\end{aligned}  
  \end{equation}
\end{ipg}

We now define an explicit contracting homotopy $\sigma$ on the rows of $\mathbb S$ that can be used to complete the argument in \ref{chunk-PL}.  
This turns out to be nothing but a scaled version of the de Rham differential. 

\begin{ipg} \label{sigmadef}
In view of Proposition~\ref{descent-dga-perturbation}, in order to transfer the dg algebra structure from $\mathbb X_a$ to $\mathbb L_a$ we need a special deformation retract, that is, we need a homotopy $h$  with the properties listed in (\ref{specialDR}). 
As explained in \ref{chunk-PL} in view of (\ref{truncated-homotopy}), this comes down to finding a contracting homotopy $\sigma$ with $\sigma^2=0$ on the rows 
\begin{equation*}
\cdots  \to \Lambda^{i-1}\otimes S_{m+1} \to \Lambda^i\otimes S_m \to \Lambda^{i+1}\otimes S_{m-1} \to \cdots
\end{equation*}
of the entire diagram $\mathbb S$ displayed in (\ref{diagramS})  with the property that $i+m>0$. 

Assume now that the field $k$ has characteristic zero (for the positive characteristic case, see the end of this portion). 

Define $\sigma_{i,m} \colon \Lambda^i\otimes S_m \to \Lambda^{i+1}\otimes S_{m-1}$ 
as
\begin{align*}
    \sigma_{i,m}&(e_{t_1}\wedge \cdots \wedge e_{t_i} \otimes y_{p_1}\cdots y_{p_m})
    &= \frac{1}{i+m} \sum_{j=1}^m e_{p_j}\wedge e_{t_1}\wedge \cdots \wedge e_{t_i} \otimes y_{p_1}\cdots \hat y_{p_j}\cdots y_{p_m}
\end{align*}
where it is understood that $\sigma_{i,m}=0$ when the target of the map is the zero module, that is, when $m=0$ or $i=n$. 
This can also be written as a scaled de Rham differential
\[
\sigma_{i,m}
=
\frac{1}{i+m} \ \sum_{j=1}^n e_{j} \otimes \frac{\partial}{\partial y_{j}}.
\]

To address the case of positive characteristic $p$, note that in general we only need to  define a contracting homotopy $\sigma_{i,m}$ for $m\leq a$ when we apply (\ref{truncated-homotopy}) to truncate the complex at position $a-1$, and so it suffices to assume $p\geq n+a$. This ensures that when necessary one has  $\frac{1}{i+m}\in k$. 

In Lemma~\ref{sigmasquared}, we show $\sigma$ is a contracting homotopy with $\sigma^2=0$. 
In view of Lemma~\ref{descent-dga}, we also require the homotopy to satisfy the generalized Leibniz property; this also comes down to the same property for $\sigma$ on the rows of the entire diagram $\mathbb S$ in (\ref{diagramS}), which we verify in Lemma~\ref{scaledLeibniz}. 

A contracting homotopy was defined previously by Srinivasan in \cite{Srin89}. However, it does not satisfy either required property. 
The map $\sigma$ defined above is more symmetric (it is invariant under permutations of the variables) and hence ends up having its square equal to zero and satisfying the generalized Leibniz rule, in fact, the stronger scaled Leibniz rule, as we see in the next two results. 
\end{ipg} 

 \begin{lem}
 \label{sigmasquared}
 Consider the maps $\sigma$ defined in (\ref{sigmadef}) on the rows of diagram (\ref{diagramS}) in which the indices sum to a positive number. 

 If $R$ has characteristic zero, the maps $\sigma$ give a contracting homotopy (that is, a null homotopy for the identity map) on the rows and satisfy $\sigma^2=0$. 
 
 If $R$ has positive characteristic $p$, the same  conclusions hold for $\sigma_{i,m}$ with $m\leq a-1$ as long as $p\geq n+a$.
 \end{lem}
 
 \begin{proof}
First we show that $\kappa \sigma + \sigma \kappa = 1$. 
At the ends of the rows one can show this easily, 
so we may work with basis elements of $\Lambda^i\otimes S_m$ with $m,i>0$. 
We compute $\kappa \sigma$ and $\sigma \kappa$ separately.
The reader should note that we do not replace any repeated factors $e_j\wedge e_j$ with $0$, noting that the formula for the Koszul differential $\kappa$ gives the same output for either form of input. 

For any $\alpha=e_{t_1}\wedge \cdots \wedge e_{t_i} \otimes y_{p_1}\cdots y_{p_m} \in \Lambda^i\otimes S_m$, 
one has 
 \begin{align*}
   &\kappa_{i+1,m-1} \sigma_{i,m}(\alpha)\\ &=\frac{1}{i+m} \sum_{j=1}^m  [\alpha +\sum_{u=1}^i(-1)^u e_{p_j}\wedge e_{t_1}\wedge \cdots \hat e_{t_u} \cdots \wedge e_{t_i} \otimes y_{t_u}y_{p_1}\cdots   \hat y_{p_j}\cdots y_{p_m}]\\
   &=\frac{1}{i+m}[m \alpha +  \sum_{j=1}^m  \sum_{u=1}^i(-1)^u e_{p_j}\wedge e_{t_1}\wedge \cdots \hat e_{t_u} \cdots \wedge e_{t_i} \otimes y_{t_u}y_{p_1}\cdots   \hat y_{p_j}\cdots y_{p_m}]
\end{align*}
 and 
  \begin{align*}
   &\sigma_{i-1,m+1}\kappa_{i,m}(\alpha)\\ 
   &= \sum_{u=1}^i(-1)^{u+1} \frac{1}{i+m}[\alpha+ \sum_{j=1}^m e_{p_j}\wedge e_{t_1}\wedge \cdots \hat e_{t_u} \cdots \wedge e_{t_i} \otimes y_{t_u}y_{p_1}\cdots   \hat y_{p_j}\cdots y_{p_m}]\\
   &=\frac{1}{i+m}[i \alpha + \! \sum_{j=1}^m  \sum_{u=1}^i(-1)^{u+1} e_{p_j}\wedge e_{t_1}\wedge \cdots \hat e_{t_u} \cdots \wedge e_{t_i} \otimes y_{t_u}y_{p_1}\!\cdots   \hat y_{p_j}\cdots y_{p_m}]
\end{align*}
Thus for $m,i>0$ one has 
\begin{align*}
    (\kappa_{i+1,m-1} \sigma_{i,m}+\sigma_{i-1,m+1} \kappa_{i,m})(\alpha)=(\frac{m}{i+m}) \alpha+ (\frac{i}{i+m}) \alpha = \alpha
\end{align*}

Next, to see that $\sigma^2=0$, one computes for $m \geq 2 $ 
\begin{align*}
   &\sigma_{i+1,m-1} \sigma_{i,m}(\alpha) =\frac{1}{(i+m)^2} \sum_{j=1}^m  \sum_{
   {u=1}\atop{u\neq j}}^m e_{p_u}\wedge e_{p_j}\wedge e_{t_1}\wedge \cdots \wedge e_{t_i} \otimes y_{p_1}\cdots \hat y_{p_u}\cdots  \hat y_{p_j}\cdots y_{p_m}
\end{align*}
which is zero since $e_{p_u}\wedge e_{p_j} = - e_{p_j}\wedge e_{p_u}$.

Note that this proof works in positive characteristic as long as the maps $\sigma_{i,m}$ are defined, which is guaranteed by the hypotheses. 
 \end{proof}
 
 \begin{lem}
 \label{scaledLeibniz}
 If $R$ has characteristic zero, the maps $\sigma$ defined in (\ref{sigmadef}) satisfy the scaled Leibniz rule. 
 More precisely, when $\alpha \in \Lambda^i \otimes S_a$ , $\beta \in \Lambda^j \otimes S_b$ with $i+a+j+b$ positive,  
 the maps $\sigma$ satisfy
 \begin{equation*}
     \sigma(\alpha \beta) = \frac{1}{i+a+j+b}\left ((i+a)\sigma(\alpha)\beta+(-1)^i(j+b)\alpha \sigma(\beta)\right )
 \end{equation*}
 and when $i=a=j=b=0$ one has that $\sigma(\alpha \beta)$, $\sigma(\alpha)$, and $\sigma(\beta)$ are  all 0. 
 
 If $R$ has positive characteristic $p$, the same conclusion holds as long as $p\geq i+a+j+b$.
 \end{lem}
 
 \begin{proof}
 Without loss of generality, we may assume that 
 \begin{equation*}
     \alpha=e_{t_1}\wedge \cdots \wedge e_{t_i} \otimes y_{p_1}\cdots y_{p_a} ~\text{and} ~
 \beta=e_{s_1}\wedge \cdots \wedge e_{s_j} \otimes y_{q_1}\cdots y_{q_b}. 
 \end{equation*}
 with $i+a+j+b>0$. 
 
 Then one has 
 \begin{align*}
     \sigma(\alpha \beta) &=
     \sigma( e_{t_1}\wedge \cdots \wedge e_{t_i}\wedge e_{s_1}\wedge \cdots \wedge e_{s_j} \otimes y_{p_1}\cdots y_{p_a} y_{q_1}\cdots y_{q_b})
     \\
    &=\frac{1}{i+a+j+b}  \left ( \sum_{u=1}^a e_{p_u}\wedge e_{t_1}\wedge \cdots \wedge e_{s_j} \otimes y_{p_1}\cdots \hat y_{p_j}\cdots y_{p_a}y_{q_1}\cdots y_{q_b}  \right.
    \\ 
     &\hspace{25mm} + \left.
     \sum_{v=1}^b e_{q_v}\wedge e_{t_1}\wedge \cdots \wedge e_{s_j} \otimes y_{p_1}\cdots y_{p_a}y_{q_1}\cdots \hat y_{q_v}\cdots y_{q_b}
    \right )
    \\
    &=\frac{1}{i+a+j+b}  \left( \sum_{u=1}^a (e_{p_u}\wedge e_{t_1}\wedge \cdots \wedge e_{t_i} \otimes y_{p_1}\cdots \hat y_{p_j}\cdots y_{p_a})(\beta)  \right.
    \\ 
    &\hspace{25mm} 
    + \left. 
    \sum_{v=1}^b (-1)^i (\alpha)  (e_{q_v}\wedge e_{s_1}\wedge \cdots \wedge e_{s_j} \otimes y_{q_1}\cdots \hat y_{q_v}\cdots y_{q_b}   )\right)
    \\
    & =\frac{1}{i+a+j+b}\left ((i+a)\sigma(\alpha)\beta+(-1)^i(j+b)\alpha \sigma(\beta)\right )
 \end{align*}
where the relevant terms are 0 when either $i+a$ or $j+b$ is 0.  
\end{proof}
 
Next we put together all the ingredients from this section to obtain our main application of our homotopy descent results. 
The proof is an application of Proposition~\ref{descent-dga-perturbation} to the deformation retract obtained in (\ref{chunk-PL}) with the homotopy $h_\infty$ defined in (\ref{h-infty}) obtained from the homotopy $\sigma$ on the rows of the diagram (\ref{diagramS}) satisfying the scaled Leibniz rule and hence the generalized Leibniz rule; see  (\ref{sigmadef}), (\ref{sigmasquared}), and (\ref{scaledLeibniz}). 
See also the overview in the paragraph before (\ref{chunk-PL}). 
One note: one need only check the homotopy $h$ on $\mathbb X_a$ satisfies the scaled Leibniz rule for products that land in $\Lambda^i\otimes S_m$ for $i\leq n$ and $m<a$ since otherwise the product is zero and the result is trivial; this explains why we need only take $p\geq a+n$ in the statement below. 

\begin{thm}
\label{thm-L}
Let $a$ be a positive integer. Suppose that $k$ is a field of characteristic zero or positive characteristic $p\geq a+n$. 
Consider the deformation retract obtained in (\ref{chunk-PL})
\[
\mathbb X_a 
\stackrel[p_\infty]{i_\infty}{\longleftrightarrows}
\mathbb L_a
\]
with the associated homotopy $h_\infty$ defined in (\ref{h-infty}) using $\sigma$ from (\ref{sigmadef}), where $i_\infty$ and $p_\infty$ are defined as in (\ref{i-infty}) and (\ref{p-infty}). 

Defining the product of  $\alpha,\beta\in \mathbb L_a$ by \begin{equation*}
     \alpha \beta 
= p_\infty 
\left( 
i_\infty(\alpha)  i_\infty(\beta) 
\right) 
  \end{equation*} 
yields a dg algebra structure on $\mathbb L_a$. Furthermore, with this structure the map $i_\infty$ is a homomorphism of dg algebras. 
 \end{thm}
 
\begin{rmk}
The product given in the theorem above can be described explicitly, using the definitions of $i_\infty$ and $p_\infty$,  as follows.  

Consider elements $\alpha,\beta\in \mathbb L_a$. 
If   one of them is in $ (\mathbb L_a)_0 = R$ then their product is the one coming from the $R$-module structure of each $(\mathbb L_a)_i$. If both have positive degree, 
then 
\[
\alpha\beta=\kappa(\widetilde{\alpha}\widetilde{\beta})
\]
where $\kappa$ is defined in (\ref{kappadef}) and 
 \begin{align*}
 \widetilde{\alpha} &= 
 (1+(-\sigma d)+(-\sigma d)^2+\cdots) \sigma  (\alpha) 
 \\   
 \widetilde{\beta} &= 
(1+(-\sigma d)+(-\sigma d)^2+\cdots) \sigma  (\beta)
 \end{align*}
 where the scaled de Rham differential $\sigma$ is defined in \ref{sigmadef}. 
\end{rmk}

\begin{rmk}
Because of the symmetric way in which the maps $\kappa$, $d$ and  $h$ are defined, the dg algebra structure defined on $\mathbb L_{a}$ in Theorem~\ref{thm-L} is invariant under the action of the symmetric group on the polynomial ring.
\end{rmk}

\begin{rmk}
\label{basisfree}
One may note that our algebra structure is, in fact, basis free, although we do not describe it in a basis free way. It is well known that the differentials in the complex from which we descend our structure are so, and one can see that the homotopy is as well, as it is just a scaled version of the de Rham map.
\end{rmk}
 \section{Comparison maps}
 \label{sec-homomorphism}

In this section we use the results from the previous sections to obtain comparison maps lifting the natural surjections 
$R/\fm^b \lra R/\fm^a$ for any $b\geq a$ to their respective minimal free resolutions $\mathbb L_b$ and $\mathbb L_a$, and these maps turn out to be dg algebra morphisms (for the dg algebra structures placed on them in the previous section). 
Since $\mathbb L_1$ is simply the Koszul complex $K$ on the variables, this yields that $K$ is a dg algebra over $\mathbb L_b$ for each $b\geq 1$.  

To set up the statement, recall from (\ref{xlhe}) that for any $c$ there is a homotopy equivalence 
\[
\mathbb X_c 
\stackrel[p_\infty]{i_\infty}{\longleftrightarrows}
\mathbb L_c
\]
pictured in (\ref{diagramXL}) that is used in Theorem~\ref{thm-L} to place a dg algebra structure on $\mathbb L_c$ for which $i_\infty$ is a dg algebra homomorphism. 
Although the value of $c$ varies below, it should be clear from the context which $i_\infty$ and $p_\infty$ maps are being applied.
Recall also from (\ref{chunk-truncation})
that we can view $\mathbb X_c$ as the quotient 
$\mathbb S / \tr_{\geq c}(\mathbb S)$
of the dg algebra $\mathbb S$ by the dg ideal 
\begin{equation*}
     \tr_{\geq c}(\mathbb S)=
    \{ \Lambda^i \otimes S_j \mid j \geq c \}
\end{equation*}
In this way, $\mathbb X_c$ inherits the dg algebra structure from $\mathbb S$. Therefore, if $b\geq a$, the inclusion of dg ideals $\tr_{\geq b}(\mathbb S) \hookrightarrow \tr_{\geq a}(\mathbb S)$ gives a natural quotient map 
\[
\pi_{b,a}
\colon
\mathbb X_b
=\mathbb S / \tr_{\geq b}(\mathbb S)
\onto 
\mathbb S / \tr_{\geq a}(\mathbb S)
=\mathbb X_a
\]
which has the effect of sending the columns $\Lambda^i\otimes S_j$ to zero for $a\leq j\leq b-1$. This is clearly a homomorphism of dg algebras. 

\begin{thm}
\label{dghomomorphism}
 Let $a$ and $b$ be positive integers with $b \geq a$. The chain map  
 \[
 f_{b,a}
 =p_{\infty}\pi_{b,a}i_{\infty}
 \colon \mathbb L_b \to \mathbb L_a
 \] 
 is a homomorphism of dg algebras that gives a lifting of the natural surjection $R/\fm^b \lra R/\fm^a$. In particular, the Koszul complex on the variables, which is $\mathbb L_1$,  is a dg algebra over $\mathbb L_b$ for every positive integer $b$. 
 
 Moreover, if $c\geq b \geq a$ then $f_{c,a}=f_{b,a} f_{c,b}$.
\end{thm}
 
\begin{proof}
First note that $f_{b,a}$ is a chain map since it is a composition of chain maps. Also, $(f_{b,a})_0$ is the identity map on $R$; thus $f_{b,a}$ gives a lifting of  the natural surjection 
$R/\fm^b \lra R/\fm^a$.

Next we show that $f_{b,a}$ is a homomorphism of dg algebras. 
Clearly, if $b=a$ then $f_{b,a}$ is the identity map. So we may assume that $b>a$. 
Let $\alpha$ and $\beta$ be homogeneous elements of $\mathbb L_b$. If either sits in degree 0, then $f_{b,a}(\alpha\beta)=f_{b,a}(\alpha)f_{b,a}(\beta)$ as $f_{b,a}$ is a homomorphism of $R$-modules. So we may assume that $\alpha\in L_{i,b}$ and $\beta\in L_{j,b}$ for some $0\leq i,j \leq n$. 
Since $\pi_{b,a}$ and $i_{\infty}$ are homomorphisms of dg algebras, one has 
\begin{align*}
    f_{b,a}(\alpha\beta)
    &=p_\infty \pi_{b,a} i_\infty (\alpha \beta) 
    \\
    &=p_\infty(\pi_{b,a} i_\infty (\alpha)  \tr_{b,a} i_\infty (\beta))
\end{align*}
We pause to compute the composition 
\begin{align}\label{tri}
\begin{aligned}
\pi_{b,a} i_\infty 
&=\pi_{b,a}(1+(h\delta) +\cdots +(h\delta)^{b-1})i
\\
&=((h\delta)^{b-a} +\cdots + (h\delta)^{b-1})i
\end{aligned}
\end{align}
and so we have 
\begin{align*}
    f_{b,a}(\alpha\beta)
    &=p_\infty\!\! 
    \left(
    [(h\delta)^{b-a} +\cdots + (h\delta)^{b-1}] i(\alpha)
    \, [(h\delta)^{b-a} +\cdots + (h\delta)^{b-1}]i(\beta)
    \right)
\end{align*}

From (\ref{tri}) we also get an alternate formula for $f_{b,a}$ as follows which we use in the next part 
\begin{align}\label{altf}
\begin{aligned}
   f_{b,a} 
   &=p_\infty \pi_{b,a} i_\infty
   \\
   &=p_\infty ((h\delta)^{b-a} +\cdots + (h\delta)^{b-1})i
   \\
   &=p_\infty (h\delta)^{b-a} i
   \end{aligned}
\end{align}
where the other terms disappear as they are in the portion of the domain where $p_\infty$ equals 0. 
Not that in the last line $p_\infty$ can be replaced by $p$.

Next we compute 
\begin{align*}
   f_{b,a}(\alpha)f_{b,a}(\beta) &=   p_{\infty}(i_{\infty}(f_{b,a}(\alpha)) \, i_{\infty}(f_{b,a}(\beta)))
\end{align*}
by the definition of the product in $\mathbb L_a$. 
This is equal to $f_{b,a}(\alpha\beta)$ because
\begin{align}\label{if}
\begin{aligned}
  i_{\infty} f_{b,a}
  &= i_{\infty} p_{\infty}(h \delta)^{b-a} i \\
  &= (1+(h\delta) +\cdots +(h\delta)^{a-1})ip[h\delta(h\delta)^{b-a-1}]i\\
  &=(1+(h\delta) +\cdots +(h\delta)^{a-1})\sigma \kappa(-\sigma) \delta(h\delta)^{b-a-1}i\\
  &=(1+(h\delta) +\cdots +(h\delta)^{a-1})(-\sigma) \delta(h\delta)^{b-a-1}i\\
   &=(1+(h\delta) +\cdots +(h\delta)^{a-1})(h\delta)^{b-a}i\\
   &=((h\delta)^{b-a} +\cdots +(h\delta)^{b-1})i
   \end{aligned}
\end{align}
where the first equality is by (\ref{altf}), the second one is by the definitions of $i_\infty$ and $p_\infty$, the third one is by the definitions of $i$, $p$, and $h$, the fourth one is because $\sigma\kappa\sigma=\sigma$ since the homotopy $\sigma$ satisfies $\sigma\kappa=1-\kappa\sigma$ and $\sigma^2=0$, and the fifth is because $h=-\sigma$. Note that when we apply the definitions of $i$, $p$, and $h$ we are using that the terms are in $\Lambda^j\otimes S$ for $j>0$. 

Last we compute the  composition 
\begin{align*}
   f_{b,a} f_{c,b}&=\left ( p_{\infty}\pi_{b,a} i_{\infty}\right ) f_{c,b} \\
   &=p_{\infty}\pi_{b,a}  ( (h\delta)^{c-b} +\cdots +(h\delta)^{c-1} )i\\
   &=p_{\infty}(h\delta)^{c-a} i = f_{c,a}
\end{align*}
where the second equality is by (\ref{if}), the third is from the definitions of the maps, and the last is by (\ref{altf}). 
as desired. 
\end{proof}
 
Note that for $a=1$, of course, the map $p_{\infty} \colon \mathbb X_1 \to \mathbb L_1$ is an isomorphism of complexes, hence it is trivial that $f_{b,a}$ is a homomorphism of dg algebras since $\pi_{b,a}$ and $i_{\infty}$ always are. 

\begin{rmk}
\label{formula}
In the proof above, the following more explicit formula for the map 
$f_{b,a} \colon \mathbb L_b \to \mathbb L_a$ 
was derived; see~\ref{altf}. 
\begin{equation}
   f_{b,a}=p (h\delta)^{b-a} i
\end{equation}
As a consequence we see that for $j>0$ 
\begin{equation}
   \im (f_{b,a})_j \subseteq \fm^{b-a} \mathbb (L_a)_j
\end{equation}
since the map $\delta=d=\kos\otimes 1$ has image in $\fm$ times the next free module as $\kos$ is the differential in the Koszul complex on the variables. 
This may also be seen in an elementary way using long exact sequences of Tor modules.  
\end{rmk}

\section*{Acknowledgements}
We thank Srikanth Iyengar for suggesting the Perturbation Lemma to us, Ben Briggs for discussions about $A_\infty$-algebras, and Luchezar Avramov for some helpful comments. 
We also Daniel Murfet for asking us questions that encouraged us to explore the connections to the Homotopy Transfer Theorem.

\bibliographystyle{amsalpha}
\bibliography{citations}

\end{document}